\documentclass[reqno,11pt]{article}
\usepackage{amsmath, amsfonts, amsthm, amscd, amssymb, commath}
\usepackage{indentfirst, bm, mathtools}
\usepackage{geometry}
\usepackage{tikz}
\usepackage{authblk, fancyhdr}
\usepackage[colorlinks=true]{hyperref}
\usepackage[justification=centering]{caption}

\usetikzlibrary{arrows, positioning, shapes, decorations.markings,decorations.pathmorphing, bending, calc}
\tikzstyle{arrowstyle}=[scale=1.5]

\newtheorem{theorem}{Theorem}[section]
\newtheorem{proposition}[theorem]{Proposition}
\newtheorem{lemma}[theorem]{Lemma}

\theoremstyle{definition}
\newtheorem{definition}[theorem]{Definition}

\theoremstyle{remark}
\newtheorem*{remark}{Remark}
\numberwithin{equation}{section}

\title{The Hopf-Tsuji-Sullivan Dichotomy on Visibility Manifolds Without Conjugate Points}

\author[1]{Fei Liu\thanks{liufei@math.pku.edu.cn}}
\affil[1]{\small College of Mathematics and Systems Science, Shandong University of Science and Technology, Qingdao, 266590, China}

\author[2]{Xiaokai Liu\thanks{12232868@mail.sustech.edu.cn}}
\affil[2]{\small Department of Mathematics, Southern University of Science and Technology, Shenzhen, 518055, China}

\author[3]{Fang Wang\thanks{~Corresponding Author: fangwang@cnu.edu.cn}}
\affil[3]{\small School of Mathematical Sciences, Capital Normal University, Beijing, 100048, China}

\begin{document}
\maketitle
\begin{abstract}
	In this article, we establish the Hopf-Tsuji-Sullivan dichotomy for geodesic flows on certain manifolds with no conjugate points: either the geodesic flow is conservative and ergodic, or it is completely dissipative and non-ergodic. We also show several equivalent conditions to the conservativity, such the Poincar\'e series diverges at the critical exponent, the conical limit set has full Patterson-Sullivan measure, etc.    


\textbf{2020~MSC}: 37D40, 37D25
\end{abstract}

\section{\bf Introduction}

Let $M$ be a smooth manifold equipped with a Riemannian metric $g$ and $\phi_t$ be the geodesic flow on the unit tangent bundle $T^1 M$. The dynamical behavior of the geodesic flow relies on the metric. It is observed that negative curvature suggests uniform hyperbolicity and non-positive curvature allows non-uniform hyperbolicity.

As a natural but significant generalization of manifolds with non-positive curvature, the manifolds with no conjugate points allows for some hyperbolic behavior of the geodesic flow, but also admits positive curvature which is an obstacle to the hyperbolicity, thus attracts great interests of dynamicists through the history. Lacking of the hyperbolicity and several powerful tools used on manifolds with negative or non-positive curvature, such as analysis involving the curvature, research under the new setting usually requires greater efforts.

Inspired by Eberlein~\cite{Eb1}, visibility axiom was introduced as a useful tool to overcome these difficulties, and fruitful results were achieved. In our article, we follow these classical ideas to advance our study. We assume that $M$ is a uniform visibility manifold with no conjugate points.

Let $\widetilde{M}$ be the universal covering space of $M$, which can be shown to be diffeomorphic to the open unit ball in $\mathbb{R}^{\text{dim} M}$ (Cartan-Hadamard Theorem). Let  $\widetilde{M}(\infty)$ denote the ideal boundary of  $\widetilde{M}$, which is defined to be the set of equivalence classes of positive asymptote geodesics on $\widetilde{M}$. 

Visibility axiom guarantees for any two different points $\xi\neq\eta$ on the ideal boundary, there is at least one connecting geodesic. In~\cite{Eb3}, Eberlein and O'Neil introduced the Axiom 2, saying that there are no more than one connecting geodesics. This property holds naturally on manifolds with negative sectional curvature, and we need it to define the Bowen-Margulis-Sullivan measure on manifolds with no conjugate points. Here we call the manifold is \emph{regular} if it satisfies the Axiom 2.

Assume $\Gamma$ is a non-elementary discrete subgroup of the isometry group of $\widetilde{M}$, which is isomorphic to $\pi_1(M)$. Then we know that  $M\cong\widetilde{M}/\Gamma$. The action of the $\Gamma$ on $\widetilde{M}$ also induces a group action on the ideal boundary $\widetilde{M}(\infty)$. Given a real number $s$ and a pair of points $p,q \in \widetilde{M}$, we can define the \emph{Poincar\'e series} of $\Gamma$ as
\begin{displaymath}
	P(s,p,q)=\sum_{\alpha\in\Gamma}\mathrm{e}^{-s d(p,\alpha q)}.
\end{displaymath}

For different $s$, the Poincar\'e series can be either convergent or divergent. The \emph{critical exponent} of this Poincar\'e series is defined as
\begin{displaymath}
	\delta_\Gamma=\inf\{s\geq 0\mid P(s,p,q)<\infty\}.
\end{displaymath}

It is observed that $\delta_\Gamma$ does not depend on the choice of $p$ or $q$. We call $\Gamma$ is of the \emph{divergence type} if $P(\delta_\Gamma,p,q)$ diverges, otherwise we call it is of the \emph{convergent type}.

A family of $\Gamma$-equivariant measures on the ideal boundary $\widetilde{M}(\infty)$, denoted by ${\{\mu_p\}}_{p\in \widetilde{M}}$, can be constructed based on the Poincar\'e series. It is named as the \emph{Patterson-Sullivan (PS) measure}. The construction of PS measure gives a concrete example of the conformal measure and links ergodic theory to harmonic analysis. Thus, this measure attracts interests through the history since it is constructed. In~\cite{Pa2}, Patterson first introduce the measure on $2$-dimensional hyperbolic manifolds, and later generalized to higher dimensional hyperbolic manifolds by Sullivan in~\cite{Su2}. Later, Yue in~\cite{Yue1} studied this measure on manifolds with pinched negative curvature; Knieper in~\cite{Kn4} and Link in~\cite{Li} discussed this family of measures on the manifolds with non-positive curvature, Ricks on rank one $\text{CAT}(0)$ spaces in~\cite{Ri2} and Das, Simmons and Urba\'{n}ski on Gromov hyperbolic metric spaces in~\cite{DDU}. 

Based on Patterson-Sullivan measure, a unique measure $m$ on $T^1M$ that is invariant under the geodesic flow $\phi_t$ can be constructed. This measure is widely used in dynamics and called the \emph{Bowen-Margulis-Sullivan (BMS) measure}.

Proven by Sullivan in~\cite{Su2, Su}, this measure is exactly the unique measure of maximal entropy for geodesic flows on manifolds with constant curvature, offering an explicit construction to support the studies of Margulis (\cite{Ma}) and Bowen (\cite{Bow}).
When the standard Hopf parametrization exists on $\widetilde{M}$ (e.g.~$M$ is a negatively curved manifold), $m$ is locally equivalent to the product measure $d\mu_p\otimes d\mu_p\otimes dt$ under the Hopf coordinates. For manifolds whose universal covering does not admit the Hopf parametrization, the BMS measure was first constructed by Knieper in~\cite{Kn4} for manifolds with non-positive curvature. 

It has been a long time since mathematicians noticed that the geodesic flow is either conservative or completely dissipative with respect to the BMS measure $m$ if certain hyperbolic properties are allowed on the manifolds. This is the so called Hopf-Tsuji-Sullivan (HTS) dichotomy. This phenomenon of dichotomy was first observed in the 1930s by Hopf in studying the dynamics of geodesic flows on surfaces with constant negative curvature, although it was several decades before the precise concept of the BMS measure was raised (see Hopf~\cite{Ho} as a good reference. In fact, Hopf originally considered the ergodicity of geodesic flows with respect to the Liouville measure, which is exactly the BMS measure in locally symmetric spaces). Later, Tsuji in~\cite{Ts} proved that conservativity is equivalent to ergodicity, and Sullivan further extended their results to higher dimensional manifolds with constant negative curvature in~\cite{Su2}. The study of this dichotomy remains active till now. Kaimanovich in~\cite{Ka2} exhibited the HTS dichotomy on Gromov hyperbolic metric spaces. Yue in~\cite{Yue1} showed the HTS dichotomy holds on manifolds with pinched negative curvature. Roblin in~\cite{Rb} proved it for the $\text{CAT}(-1)$ metric spaces. Link and Picaud in~\cite{Li2, LP} for the non-positively curved spaces. Recently, Burger, Landesberg, Lee and Oh in~\cite{BLLO} establish an extension of the HTS dichotomy to any Zariski dense discrete subgroup of a semisimple real algebraic group.

Furthermore, it is observed that the conservativity and ergodicity of the geodesic flow is highly related to the ergodicity of the $\Gamma$-action on $\widetilde{M}(\infty)\times\widetilde{M}(\infty)$, with respect to the BMS measure and PS measure respectively. In this paper, we prove this observation on complete, regular and uniform visibility manifolds with no conjugate points. We state our main theorem as the following:

\begin{theorem}[Main Theorem]
	Let $M$ be a complete, regular and uniform visibility manifold with no conjugate points, $\mu={\{\mu_q\}}_{q\in\widetilde{M}}$ be the Patterson-Sullivan measure on $\widetilde{M}(\infty)$. Fix any point $p\in\widetilde{M}$, and let $m$ be the BMS measure on $T^1M$ induced by $\mu_p$. Then either the geodesic flow $\phi_t$ is conservative or completely dissipative with respect to $m$ (namely, the HTS dichotomy holds). Moreover, the following statements are equivalent:\vspace{-1.5ex}
	\begin{itemize}
		\item The conical limit set $L_c(\Gamma)$ has full $\mu_p$ measure in $\widetilde{M}(\infty)$.
		\item $\phi_t$ is conservative with respect to $m$.
		\item $\phi_t$ is ergodic with respect to $m$.
		\item The $\Gamma$-action on $\widetilde{M}(\infty)\times\widetilde{M}(\infty)$ is ergodic with respect to $\mu_p\times\mu_p$.
		\item The Poincar\'e series $\sum_{\alpha\in\Gamma}\mathrm{e}^{-\delta_\Gamma d(p,\alpha p)}$ diverges.
	\end{itemize}
\end{theorem}

The conical limit set $L_c(\Gamma)$ is a subset of the limit set of $\Gamma$, denoted by $L(\Gamma)$. A key observation is that, the set of conservative points on $T^1M$, say $M_{C}$, can be characterized by their future limit points lying in $L_c(\Gamma)\subset L(\Gamma)$ (Proposition~\ref{prop_6_2}). This reveals the relation between the dynamics of $\phi_t$ on $T^1M$ and the $\Gamma$-action on the ideal boundary.

In establishing our main theorem, the shadow lemma of the PS measure plays a key role. The shadow lemma shows that the $\mu_p$-measure of the projection balls $\text{pr}_p(B(\alpha p,r))\subset \widetilde{M}(\infty)$ is of $O(e^{-\delta_{\Gamma}d(\alpha p, p)})$ as $d(\alpha p, p)\rightarrow \infty$. This estimation (known as Sullivan shadow lemma) and its generalizations (such as Mohsen shadow lemma) was first proven on manifolds with negative curvature, and further extended to more general cases (cf.~\cite{Su2, PPS, Kn4, Yue1}). In this paper, we proved the shadow lemmas on a more general kind of manifolds. As a key tool in the quantitative analysis related to the PS measure, the shadow lemma is widely used, but not limited to, the HTS dichotomy.

We remark that the ultimate goal is to study if the HTS dichotomy holds on manifolds with no conjugate points, which remains open for decades if no extra condition added. But if we narrow down to the case of no focal points, we could get the results under much milder conditions. With the help of Flat Strip Theorem, we are able to define the BMS measure without the need of manifolds being regular. Carefully dealing with the flat strips (if any), we establish the HTS dichotomy on visibility manifolds with no conjugate points, under a technical condition that the width of these flat strips has a positive lower bound. 

This article is organized in the following way: In Section~\ref{sec2}, we give the definitions of basic concepts and show some basic geometric properties, as the preparation of further discussions. Section~\ref{sec3} focuses on the conformal density and construct the PS measure. The Poincar\'e series and the critical exponent $\delta_{\Gamma}$ will also be discussed there.  The Sullivan shadow lemma and the Mohsen shadow lemma under our settings are proven in Section~\ref{sec4}.  Based on the shadow lemma, the equivalence relation between the divergence of the Poincar\'e series at $\delta_{\Gamma}$ and the ergodicity of the $\Gamma$-action on $\widetilde{M}(\infty)$ with respect to the PS measure is proven in Section~\ref{sec5}. The main theorem is proved in Section~\ref{sec6}. In Section~\ref{sec7}, we adapt our results to the manifolds with no focal points under weaker conditions.

\section{\bf Basic Geometric Properties}\label{sec2}
Let $(M,g)$ be a Riemannian manifold and $(\widetilde{M},\widetilde{g})$ be its universal cover. We abuse the notations a little and use $d$ to denote the distance function induced by both $g$ on $M$ and $\widetilde{g}$ on $\widetilde{M}$ when there is no confusion. Denoted the geodesic flows on $M$ by $c(t)$, and we always assume they are of unit speed throughout the article. Let $\Gamma=\pi_1(M)$, which is a discrete subgroup of the isometric group of $\widetilde{M}$. We know that $M=\widetilde{M}/\Gamma$.

Let $T^1 M$ and $T^1\widetilde{M}$ denote the unit tangent bundles with the standard projection map $\pi: T^1M\to M$ and $\pi: T^1\widetilde{M}\to \widetilde{M}$. We use $\pi$ for both when there is no confusion.

\begin{definition}
    Let $c=c(t)$ be a geodesic in $M$. We say two points $p_0=c(t_0)$ and $p_1=c(t_1)$ on the geodesic $c$ are \emph{conjugate}, if there is a non-trivial Jacobi field $\bm{J}$, such that
	\begin{displaymath}
        \bm{J}(t_0)=\bm{J}(t_1)=0.
	\end{displaymath}
\end{definition}

We call $M$ is a \emph{manifold with no conjugate points} if no geodesic on $M$ admits conjugate points. A very straightforward example of a pair of conjugate points are the antipodal points on the standard sphere. By definition, manifolds with non-positive curvature have no conjugate points. But no conjugate points does not necessarily imply non-positive curvature. In fact, there are examples of manifolds with no conjugate points admits positive sectional curvature (cf.~\cite{Gu}).

In study the manifold with no conjugate points, visibility is often assumed to guarantee some important geometric properties.

\begin{definition}[Visibility]
    Say $\widetilde{M}$ satisfies the \emph{visibility axiom} if for any $p\in\widetilde{M}$, $\epsilon>0$, there exists $R=R(p,\epsilon)>0$, such that for any geodesic (segment) $c:[a,b] \to \widetilde{M}$ with $d(p,c)\geq R$, we have $\measuredangle_p(c(a),c(b))\leq\epsilon$. Here we allow $a$ and $b$ to be infinity.

    Say $\widetilde{M}$ satisfies the \emph{uniform visibility axiom} if the constant $R$ can be chosen independently of the choice of $p$. Say $M$ satisfies the (uniform) visibility axiom if its universal cover $\widetilde{M}$ does. We also call such manifold a (uniform) visibility manifold.
\end{definition}

The introduction of visibility is very natural. It is proved that manifolds with negative sectional curvature always satisfy the visibility axiom as shown in the following proposition.

\begin{proposition}[\cite{BO}]
	Let $M$ be a complete, simply-connected manifold with sectional curvature $\kappa\leq a <0$ for some negative constant $a$. Then $M$ satisfies the uniform visibility axiom.
\end{proposition}

But manifolds that satisfy visibility axiom may not necessarily have negative sectional curvature. There are actually a bunch examples of such manifolds. For example,

\begin{theorem}[\cite{Eb1}]
	A closed surface with no conjugate points and genus no less than $2$ satisfies the uniform visibility axiom.\label{thm_2_4}
\end{theorem}

There are also examples of homogeneous manifolds in higher dimension that satisfy visibility axiom but has zero sectional curvature at every point. Refer to Remark 1.8.4 in~\cite{Eb4} for more details. Hence, we can see visibility manifolds are actually non-trivial generalization of manifolds with negative sectional curvature.

Next, we will introduce the construction of the ideal boundary and the cone topology.

\begin{definition}
	Let $c_1$ and $c_2$ be two geodesics in $\widetilde{M}$.
	\begin{itemize}
		\item We say they are \emph{positively asymptotic}, if there is a positive constant $D$, such that
		      \begin{displaymath}
			      d(c_1(t),c_2(t))\leq D, \quad\forall t\geq 0
		      \end{displaymath}
		\item We say they are \emph{negatively asymptotic}, if there is a positive constant $D$, such that
		      \begin{displaymath}
			      d(c_1(t),c_2(t))\leq D, \quad\forall t\leq 0
		      \end{displaymath}
		\item We say they are \emph{bi-asymptotic}, if they are both positively and negatively asymptote.
	\end{itemize}
\end{definition}

It is clear that the positively/negatively asymptote constitute an equivalence relation among geodesics on $\widetilde{M}$. Let $c(+\infty)$ and $c(-\infty)$ denote the equivalence classes containing the geodesics positively/negatively asymptote to $c$. We call such set of equivalence classes the points at infinity and denote it by $\widetilde{M}(\infty)$. This set is the ideal boundary of the universal cover.

\begin{proposition}[\cite{Eb1}]\label{prop_2_6}
    Let $\widetilde{M}$ be a visibility manifold with no conjugate points. Then for any point $p\in\widetilde{M}$ and $\xi\in\widetilde{M}(\infty)$, there is a unique connecting geodesic $c_{p,\xi}$ with 
    \begin{displaymath}
        c_{p,\xi}(0)=p,\quad c_{p,\xi}(+\infty)=\xi.
    \end{displaymath}
\end{proposition}

Here connecting $\xi$ means the geodesic $c_{p,\xi}$ is positively asymptotic to $\xi$. Therefore, given two different vectors $\bm{v}_1,\bm{v}_2\in T_p^1\widetilde{M}$, we have $\lim\limits_{t\to\infty}d(c_{\bm{v}_1}(t),c_{\bm{v}_2}(t))=\infty$, where $c_{\bm{v}}{(t)}$ denotes the geodesic with $c'_{\bm{v}}(0)=\bm{v}$. This suggests that the ideal boundary is diffeomorphic to $T_p^1\widetilde{M}$. In fact, by Cartan-Hadamard Theorem, $\widetilde{M}$ is diffeomorphic to $\mathbb{R}^{\text{dim}M}$ and the ideal boundary $\widetilde{M}(\infty)$ homeomorphic to $\mathbb{S}^{\text{dim}M-1}$.

To elaborate this homeomorphism, we introduce some notations here and construct the topology on the ideal boundary, which is known as the cone topology.

Give a point $p\in\widetilde{M}$ and let $c_{p,q}$ be the unique geodesic connecting $p$ and $q$ with $c_{p,q}(0)=p$. We measure the angles between geodesics as following:

\begin{displaymath}
	\begin{aligned}
		\measuredangle_p(x,y)      & =\measuredangle(c'_{p,x}(0),c'_{p,y}(0)),\quad x,y\neq p,                        \\
		\measuredangle_p(A)        & =\sup\limits_{x,y\in A}\measuredangle_p(x,y),                                    \\
		\measuredangle_p(x,\bm{v}) & =\measuredangle(c'_{p,x}(0),\bm{v}),\quad \bm{v}\in T_p^1\widetilde{M}, x\neq p. \\
	\end{aligned}
\end{displaymath}

Next define the set of $\widetilde{M}\cup\widetilde{M}(\infty)$ concerning the angles. Let $\bm{v}\in T_p^1\widetilde{M}$,
\begin{displaymath}
	\begin{aligned}
		C(\bm{v},\epsilon)    & =\{x\neq p\mid \measuredangle_p(x,\bm{v})<\epsilon\},                                                                         \\
		C_\epsilon(\bm{v})    & =\{c_{\bm{w}}(+\infty)\mid \bm{w}\in T^1_p\widetilde{M},\measuredangle(\bm{v},\bm{w})<\epsilon\}\subset\widetilde{M}(\infty), \\
		TC(\bm{v},\epsilon,r) & =\{x\in\widetilde{M}\cup\widetilde{M}(\infty)\mid \measuredangle_p(x,\bm{v})<\epsilon, d(p,x)> r\}.
	\end{aligned}
\end{displaymath}

The last notation $TC(\bm{v},\epsilon,r)$ is called the \emph{truncated cone with axis $\bm{v}$ and angle $\epsilon$}. By the definition, we can see the point on the ideal boundary $\xi=\gamma_{\bm{v}}(+\infty)\in TC(\bm{v},\epsilon,r)$. In general, for any point $\xi\in\widetilde{M}(\infty)$, the set of truncated cones containing this point actually forms local bases and hence, forms the bases for a topology $\tau$. This topology is unique and usually called the cone topology. Under this topology, $\widetilde{M}$ is homeomorphic to $\mathbb{R}^{\text{dim}M}$ and the ideal boundary $\widetilde{M}(\infty)$ is homeomorphic to $\mathbb{S}^{\text{dim}M-1}$. For a more detailed discussion on this topology, refer to~\cite{Eb3}.

On visibility manifolds with no conjugate points, we have the following geometric properties needed:

\begin{proposition}\label{prop_2_7}
	Assume $\widetilde{M}$ to be the universal cover of a uniform visibility manifold $M$ with no conjugate points.
	\begin{enumerate}
		\item\cite{LW} For any two positive asymptote geodesic $c_1$ and $c_2$, we have the distance between is bounded:
		\begin{displaymath}
		    d(c_1(t),c_2(t))\leq 2R(\pi/2)+3d(c_1(0),c_2(0)).	
		\end{displaymath}
        Here the $R(\frac{\pi}{2})$ is the uniform visibility constant.
		\item\cite{Eb1} The following map is continuous:
		\begin{displaymath}
			\begin{aligned}
				\Psi:~& T^1\widetilde{M}\times[-\infty,\infty] & \to     &~\widetilde{M}\cup\widetilde{M}(\infty), \\
				       & \qquad(\bm{v},t)                       & \mapsto &~c_{\bm{v}}(t).
			\end{aligned}
		\end{displaymath}
    \item For any $\bm{v}\in T^1{\widetilde{M}}$ and $R,\epsilon$ positive, there is a constant $L=L(\epsilon, R)$, such that for any $t>L$, we have
        \begin{displaymath}
            B(c_{\bm{v}}(t),R)\subset C(\bm{v},\epsilon).
        \end{displaymath}
        Here $B(c_{\bm{v}}(t),R)$ is the open ball centered at $c_{\bm{v}}(t)$ with radius $R$.
    \item\cite{Eb3} For any two points $\xi\neq\eta$ on the ideal boundary, there is at least one connecting geodesic of them. 
	\end{enumerate}
\end{proposition}

Proposition~\ref{prop_2_7} (3) follows from the uniform visibility condition. And the constant $L$ here does not depend on the choice of $\bm{v}$.

Next, we introduce the Busemann function.

\begin{definition}[Busemann Function]
	We call the following map the Busemann function (or Busemann co-cycle).
	\begin{displaymath}
		\begin{aligned}
			\beta:~&~\widetilde{M}(\infty)\times\widetilde{M}\times\widetilde{M} & \to     &~\mathbb{R}                                               , \\
			        & \qquad(\xi,p,x)                                              & \mapsto &~\beta_{\xi}(p,x)=\lim_{t\to\infty}(d(p,c_{x,\xi}(t))-t).
		\end{aligned}
	\end{displaymath}
\end{definition}

Easy to check that $|\beta_{\xi}(p,x)|\leq d(p,x)$ by the triangle inequality.

\begin{definition}[Horosphere]
	We call the set
	\begin{displaymath}
		H_{\xi}(p)=\{x\in\widetilde{M}\mid\beta_{\xi}(p,x)=0\}
	\end{displaymath}
	the horosphere with center $\xi\in\widetilde{M}(\infty)$, based at $p\in\widetilde{M}$. We can see $\beta_{\xi}(p,x)$ is exactly
	the algebraic length of the part of geodesics bounded between $H_{\xi}(p)$ and $H_{\xi}(x)$ with the common end point $\xi$.
\end{definition}

The last assumption we need for the manifolds we study is that for any two points $\xi\neq\eta$ on the ideal boundary, there is \emph{at most one} geodesic connecting them. This condition holds for manifolds with negative sectional curvature, but also holds under weaker conditions. In fact, this is a natural replacement of the curvature condition. Refer to~\cite{Eb3} for a detailed discussion of this condition, in which it is called the \emph{Axiom 2}. For convenience, we call such manifolds \emph{regular}.

\begin{definition}[\cite{Eb1, Eb3}]
    We call a complete simply connected Riemannian manifolds $\widetilde{M}$ \emph{Regular}, if for any two points $\xi\neq\eta$ on the ideal boundary, there is at most one geodesic connecting them. 
\end{definition}

Together with Proposition~\ref{prop_2_7}~(4), we can argue that for regular manifolds, any two different points $\xi\neq\eta$ on the ideal boundary, there is a unique connecting geodesic.

Let $\widetilde{M}^2(\infty)=\widetilde{M}(\infty)\times\widetilde{M}(\infty)-\{(\xi,\xi)\mid \xi\in\widetilde{M}(\infty)\}$. Given a point $p\in\widetilde{M}$, we define
\begin{displaymath}
	\begin{aligned}
		\beta_p:~& \widetilde{M}^2(\infty) & \to     &~\mathbb{R}^{+},                     \\
		          & (\xi,\eta)              & \mapsto &~\beta_{\xi}(p,x)+\beta_{\eta}(p,x). \\
	\end{aligned}
\end{displaymath}
Here $x\in\widetilde{M}$ is any point on the connecting geodesic $c_{\xi,\eta}$.

No hard to prove that this function does not depend on the choice of $c_{\xi,\eta}$. Nor does $\beta_p$ depend on the choice of $x$. Geometrically, $\beta_p(\xi,\eta)$ is the length of $c_{\xi,\eta}$ bounded between the horospheres $H_{\xi}(p)$ and $H_{\eta}(p)$ as shown in Figure~\ref{fig1}.

\begin{figure}[htbp]
\centering
	\scalebox{0.7}{
		\begin{tikzpicture}[]
			\draw (0,0) circle (4);
			\coordinate [label=left:$\xi$] (X) at (-4,0);
			\node at (X) [circle, fill, inner sep=1pt] {};
			\coordinate [label=right:$\eta$] (Y) at (4,0);
			\node at (Y) [circle, fill, inner sep=1pt] {};
			\draw [->] (X) to [in=170, out=10](Y);
			\draw [blue] (-1,0) circle (3);
			\draw [green] (2,0) circle (2);
			\coordinate [label=above:$p$] (P) at (1.34, 1.88);
			\node at (P) [circle, fill, inner sep=1pt] {};
			\coordinate [label=below:$x$] (Q) at (2.5, 0.24);
			\node at (Q) [circle, fill, inner sep=1pt] {};
			\node at (1,0) [red] {$\scriptstyle\beta_p(\xi,\eta)$};
			\draw [red, very thick] (0.05,0.4) to [in=175, out=1] (1.99,0.30);
		\end{tikzpicture}
	}
	\caption{Illustration of $\beta_p(\xi,\eta)$.}\label{fig1}
\end{figure}
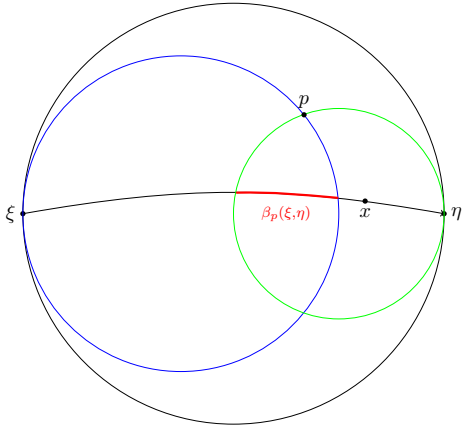

\begin{definition}[Gromov-Bourdon Visual Distance]
	The above defined function $\beta_p$ is called the Gromov product, which induces the Gromov-Bourdon visual distance on the ideal boundary by
	\begin{displaymath}
		d_p(\xi,\eta)=\mathrm{e}^{-\frac 12\beta_p(\xi,\eta)}.
	\end{displaymath}
\end{definition}
\section{\bf Conformal Densities}\label{sec3}
Let $M$ be a complete Riemannian manifold with no conjugate points, $\widetilde{M}$ be its universal cover. Let $\Gamma$ be a discrete subgroup of the isometry group of $\widetilde{M}$ and $p$ be a point in $\widetilde{M}$.

\begin{definition}
	Denote $L(\Gamma)=\widetilde{M}(\infty)\cap\overline{\Gamma(p)}$ where $\overline{\Gamma(p)}$ is the closure of the orbit of $p$ under $\Gamma$ in $\widetilde{M}\cup\widetilde{M}(\infty)$ under the cone topology. We call $L(\Gamma)$ the \emph{limit set} of $\Gamma$. If $\#L(\Gamma)=\infty$, we call $\Gamma$ is non-elementary.
\end{definition}

By Proposition~\ref{prop_2_7}~(3), we can see $L(\Gamma)$ is independent of the choice of $p$.

Next we will introduce the \emph{conformal density}, or \emph{Busemann density}, based on the \emph{Poincar\'e series} with its \emph{critical exponent}.

\begin{definition}
    For a given positive constant $r$, a family of finite Borel measure ${\{\mu_p\}}_{p\in\widetilde{M}}$ on the ideal boundary is called an \emph{$r$-dimensional conformal density} or \emph{Busemann density}, if
	\begin{enumerate}
		\item For each $p\in\widetilde{M}$, the support of $\mu_p$ is contained in $L(\Gamma)$.
		\item For any $p\neq q\in\widetilde{M}$ and $\mu_p$-a.e. $\xi\in\widetilde{M}(\infty)$ on the ideal boundary, we have
		      \begin{displaymath}
			      \frac{\mathrm{d}\mu_p}{\mathrm{d}\mu_q}(\xi)=\mathrm{e}^{-r\beta_{\xi}(p,q)},
		      \end{displaymath}
		      where $\beta_{\xi}(p,q)$ is the Busemann functions defined above.
		\item ${\{\mu_p\}}_{p\in\widetilde{M}}$ is $\Gamma$-equivariant. That is to say, for any Borel subsets $A\subset\widetilde{M}(\infty)$ and $\alpha\in\Gamma$,
		      \begin{displaymath}
			      \mu_{\alpha p}(\alpha A)=\mu_p(A).
		      \end{displaymath}
	\end{enumerate}
\end{definition}

\begin{definition}
	Given a real number $s$ and a pair of point $p,q$ in $\widetilde{M}$, we define the \emph{Poincar\'e Series} as
	\begin{displaymath}
		P(s,p,q)=\sum_{\alpha\in\Gamma}\mathrm{e}^{-s d(p,\alpha q)}.
	\end{displaymath}
	We define the \emph{critical exponent} of this Poincar\'e series as
	\begin{displaymath}
		\delta_\Gamma=\inf\{s\geq 0\mid P(s,p,q)<\infty\}.
	\end{displaymath}
\end{definition}
Easy to check that $\delta_\Gamma$ does not depend on the choice of $p$ or $q$. We call $\Gamma$ is of \emph{divergence type} if $P(\delta_\Gamma,p,q)$ diverges.

In general, given $\Gamma$, it is not easy to determine if the Poincar\'e series converges or diverges at the critical exponent, but Patterson~\cite{Pa2} gave a clever method to skirt around. He constructed a positive, monotonic increasing function $h$ defined on $\mathbb{R}^{+}$, and define
\begin{displaymath}
	\widetilde{P}(s,p,q)=\sum_{\alpha\in\Gamma}h(d(p,\alpha q))\mathrm{e}^{-s d(p,\alpha q)}.
\end{displaymath}

It can be shown that the critical exponent for $\widetilde{P}$ is the same as the one of $P$, and $\widetilde{P}$ is the divergent type. Therefore, we can always assume $\Gamma$ is of divergent type.

Now fix a point $q\in\widetilde{M}$, for any $s>\delta_\Gamma$ and $p\in\widetilde{M}$, we can define the measure
\begin{displaymath}
	\mu_{p,q,s}=\frac{\sum_{\alpha\in\Gamma}\mathrm{e}^{-s d(p,\alpha q)}\delta_{\alpha q}}{P(s,q,q)}.
\end{displaymath}

Here $\delta_{\alpha q}$ is the Dirac measure supported on the point $\alpha q$. From the triangle inequality, we can show that
\begin{enumerate}
	\item $\mu_{p,q,s}(\widetilde{M}\cup\widetilde{M}(\infty))\in[\mathrm{e}^{-s d(p,q)},\mathrm{e}^{s d(p,q)}]$.
	\item $\Gamma(q)\subset \text{supp}(\mu_{p,q,s})\subset \overline{\Gamma(q)}$.
\end{enumerate}

For any $p\in\widetilde{M}$ and a sequence of numbers ${\{s_n\}}_{n=1}^\infty\subset\mathbb{R}^{+}$ with $s_n$ decreasing to $\delta_\Gamma$, the corresponding measures $\mu_{p,q,s_n}$ will converge to some measure in weak-$*$ topology, which is denoted by $\mu_p$:
\begin{displaymath}
	\lim_{s_n\to\delta_\Gamma^+}\mu_{p,q,s_n}\overset{w^\ast}{=}\mu_p.
\end{displaymath}

By definition, the support of each $\mu_{p,q,s_n}$ is on $\{\Gamma q\}$. As we assume $\Gamma$ is of divergent type, $P(s,q,q)$ diverges when $s=\delta_\Gamma$. Together with that the action $\Gamma$ on $\widetilde{M}$ is discrete, after taking the limit, the support of $\mu_p$ will not be on the orbit of $q$ under $\Gamma$, but will be `pushed' to the limit set on ideal boundary. In the end, we conclude $\text{supp}(\mu_p)\subset \overline{\Gamma(q)}\cap\widetilde{M}(\infty)$.
\begin{theorem}
	For any point $p\in\widetilde{M}$, the following weak-$*$ limit exist:
	\begin{displaymath}
		\lim_{n\to\infty}\mu_{p,q,s_n}\overset{w^*}{=}\mu_p.
	\end{displaymath}

This family of measures ${\{\mu_p\}}_{p\in\widetilde{M}}$ is a $\delta_\Gamma$-dimensional conformal density.
\end{theorem}

In fact, in~\cite{Su2} Sullivan showed that the dimension of the limit set, or the support of $\mu_p$, is exactly $\delta_\Gamma$. A $\delta_\Gamma$-dimensional conformal density constructed in this way is called a \emph{Patterson-Sullivan (PS) measure}.  Roughly speaking, the PS measure is a realization of the conformal density at the critical exponent. As a remark, we need to point out that the definition does not guarantee the uniqueness of the weak-$*$ limit. The different choices of the sequence ${\{s_n\}}_{n=1}^\infty$ may result in different weak-$*$ limits. But the uniqueness holds for a rich type of manifolds, like rank 1 manifolds with non-positive curvature~\cite{Kn4}. In Section~\ref{sec5}, we will show the uniqueness of the PS measure in our setting in Proposition~\ref{prop_5_7}.

Another thing need to mention is that although $\Gamma$ is an isometry group acting on $\widetilde{M}$, we can extend the action to $\widetilde{M}(\infty)$ without any difficulties. Given any $\alpha\in\Gamma$ and $\xi\in\widetilde{M}(\infty)$ with any geodesic $c$ in $\widetilde{M}$ with $c(+\infty)=\xi$, we can define
\begin{displaymath}
	\alpha(\xi)=(\alpha\circ c)(+\infty).
\end{displaymath}
Easy to check, this action does not depend on the choice of $c$ and induce a homeomorphism of $\widetilde{M}(\infty)$ under the cone topology.

This enables us to talk about the ergodicity of the action of $\Gamma$ with respect to an $r$-dimensional conformal density $\mu={\{\mu_p\}}_{p\in\widetilde{M}}$.

\begin{definition}
	We call $\Gamma$ is \emph{ergodic} with respect to an r-dimensional density $\mu$, if for any $\Gamma$-invariant Borel set $A\subset\widetilde{M}(\infty)$, there exists a point $p\in\widetilde{M}$ (and hence for every point in $\widetilde{M}$), such that $\mu_p(A)\cdot\mu_p(A^{c})=0$.
\end{definition}

A well-known result by Yue~\cite{Yue1} shows that the ergodicity is equivalent to the uniqueness of the conformal density, up to a positive multiplier.

\begin{proposition}\label{prop_3_6}
	Let $\mu={\{\mu_p\}}_{p\in\widetilde{M}}$ is an $r$-dimensional conformal density on the ideal boundary $\widetilde{M}(\infty)$, and $C(r)$ is the collection of all $r$-dimensional conformal densities on $\widetilde{M}(\infty)$. Then we have
	\begin{displaymath}
		C(r)=\{\lambda\mu \mid \lambda>0\} \quad \text{if and only if}\quad \Gamma~\text{is ergodic with respect to } \mu.
	\end{displaymath}
\end{proposition}

\section{\bf The Shadow Lemma}\label{sec4}
In this section, we are going to talk about the shadow lemma as the preparation of the proof to our main theorem. We start with the definition of the shadow.

\begin{definition}
	For any pair of point $p,q\in\widetilde{M}$ and $r>0$, define the shadow of the ball $B(q,r)$ from $p$ is a subset of $\widetilde{M}(\infty)$:
	\begin{displaymath}
		\text{pr}_p(B(q,r))=\{c_{p,z}(+\infty)\mid z\in B(q,r)\}.
	\end{displaymath}
\end{definition}

\begin{lemma}\label{lem_4_2}
	For any pair of points $p,q\in\widetilde{M}$, $r>0$ and any $\xi\in\text{pr}_p(B(q,r))$, we have
	\begin{displaymath}
		d(p,q)-2r\leq\beta_\xi(p,q)\leq d(p,q).
	\end{displaymath}
\end{lemma}

\begin{proof}
	$\beta_\xi(p,q)\leq d(p,q)$ follows the triangle inequality.

	For the other inequality, pick $x\in c_{p,\xi}\cap B(q,r)$. Let $y=c_{p,\xi}\cap H_{\xi}(q)$. As shown in Figure~\ref{fig2} we have $d(x,y)\leq d(x,q)<r$, so $d(q,y)\leq d(q,x)+d(x,y)<2r$.

	\begin{figure}[htbp]
		\centering
		\scalebox{0.7}{
			\begin{tikzpicture}
				\draw (0,0) circle (4);
				\coordinate [label=right:$\xi$] (Z) at (4,0);
				\node at (Z) [circle,fill,inner sep=1pt]{};
				\draw [dashed] (4,4) to (4,-4);
				\node at (4.5,-2.7) {$\scriptstyle\widetilde{M}(\infty)$};
				\draw[red] (1.8,0) circle (2.2);
				\draw[blue] (1.5,0) circle(2.5);
				\node at (0,-4.3) {$\scriptstyle H_\xi(p)$};
				\node [blue] at (0.5,-2.7) {$\scriptstyle H_\xi(q)$};
				\node [red] at(1.5,1.8) {$\scriptstyle H_\xi(x)$};
				\coordinate [label=above:$p$] (P) at (-2,-3.47);
				\draw[thick, ->] (P) to [in=180, out=60](Z);
				\draw[thick] (P) to [out=240, in=60] (-2.3,-4);
				\node at (P) [circle, fill, inner sep=1pt] {};
				\coordinate [label=below:$\scriptstyle y$] (Y) at (-0.45,-1.55);
				\node at (Y) [blue, circle, fill, inner sep=1pt] {};
				\coordinate [label=below:$\scriptstyle x$] (X) at (-0.05,-1.2);
				\node at (X) [red,circle, fill, inner sep=1pt] {};
				\node at (-1,-3) {$c_{p,\xi}$};
				\coordinate [label=right:$q$] (Q) at (-0.75,-1.1);
				\node at (Q) [blue, circle, fill, inner sep=1pt] {};
				\draw[black] (Q) circle [x radius=1cm, y radius=0.3cm];
				\draw[fill, green, opacity=0.3] (Q) circle [x radius=1cm, y radius=0.3cm];
				\node at (-2.3,-1) {$\scriptstyle B(q,r)$};
			\end{tikzpicture}
		}
		\caption{Bound for the Busemann Function.}\label{fig2}
	\end{figure}
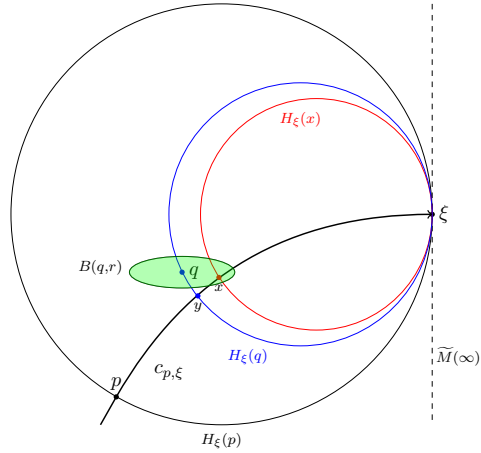

	Thus, we have
	\begin{displaymath}
		\beta_{\xi}(p,q)=\beta_{\xi}(p,y)=d(p,y)\geq d(p,q)-d(q,y)\geq d(p,q)-2r.
	\end{displaymath}
\end{proof}

Next we will show a very intuitive proposition of the shadow, that is to say, larger balls have larger shadow on the ideal boundary.

\begin{proposition}\label{prop_4_3}
	Assume $\widetilde{M}$ is a simply connected and visibility manifold with no conjugate points. For any $p\in\widetilde{M}$ and any $\epsilon>0$, there is an $R=R(p,\epsilon)>0$, such that for any $r>R$ and $q\in\widetilde{M}-B(p,r)$, the diameter of $\widetilde{M}(\infty)-\text{pr}_q(B(p,r))$ is less than $\epsilon$ with respect to the distance $d_p$ defined on $\widetilde{M}(\infty)$.
\end{proposition}

\begin{proof}
	Assume the contrary. We can find $p\in\widetilde{M}$ and $\delta>0$, with ${\{q_n\}}_{n=1}^\infty\subset\widetilde{M}$ and ${\{r_n\}}_{n=1}^\infty\subset\mathbb{R}^+$, such that $\lim\limits_{n\to\infty}r_n=+\infty$, $d(p,q_n)>r_n$ and there are $\xi_n\in\widetilde{M}-\text{pr}_{q_n}(B(p,r_n))$ such that $d_p(\xi_n,c_{p,q_n}(+\infty))\geq\delta>0$.

	By the compactness of the ideal boundary $\widetilde{M}(\infty)$, we can assume $\xi_n\to\xi$ is a converging sequence. Also, by the compactness of $\overline{\widetilde{M}}=\widetilde{M}\cup\widetilde{M}(\infty)\approx \mathbb{B}^{\text{dim}\widetilde{M}}$ under the cone topology, by choose a sub-sequence if needed, we can also assume that $q_n\to\eta\in\widetilde{M}(\infty)$ is a converging sequence as well, leading to the angle $\measuredangle_p(q_n,\eta)\to 0$.

	Now let $\bm{v}_n=c'_{p,q_n}(0)$ and $\bm{v}=c'_{p,q}(0)$ be the initial vector of the geodesics. From Proposition~\ref{prop_2_7}, we have the continuity of the terminal point of the geodesics as
	\begin{displaymath}
		\eta=c_{\bm{v}}(+\infty)=\lim_{n\to\infty}c_{\bm{v}_n}(+\infty).
	\end{displaymath}

	By the assumption, we have $d_p(\xi,\eta)\geq\delta>0$, thus $\xi\neq\eta$.

	But on the other hand, $\xi_n\in\widetilde{M}(\infty)-\text{pr}_{q_n}(B(p,r_n))$, we know the geodesic rays $c_{q_n,\xi_n}$ does not intersect with $B(p,r_n)$. Let $\theta_n=\measuredangle_p(q_n,\xi_n)$ denote the angle between $c_{p,q_n}$ and $c_{p,\xi_n}$. Because $\widetilde{M}$ is a visibility manifold, and $r_n$ diverges to $+\infty$, this angle $\theta_n\to 0$ as $n\to\infty$. Without loss of generality, by choosing sub-sequences of $q_n\to\eta$ and $\xi_n\to\xi$, we can assume $\theta_n<\frac 1n$, $\measuredangle_p(\xi_n,\xi)<\frac 1n$ and $\measuredangle_p(q_n,\eta)<\frac 1n$.

	\begin{figure}[htbp]
		\centering
		\scalebox{0.7}{
			\begin{tikzpicture}
				\coordinate[label=above:$p$](P) at (0,0);
				\node at (P) [circle, fill, inner sep=1pt] {};
				\draw (0,0) circle (4);
				\coordinate (T1) at (4,0);
				\coordinate (T2) at (-3.464,-2);
				\draw[red, thick] (T1) arc (0:210:4);
				\coordinate[label=below left:$\xi_n$] (X) at (-2,-3.464);
				\node at (X) [circle, fill, inner sep=1pt] {};
				\coordinate[label=below:$\xi$] (Y) at (-1,-3.87);
				\node at (Y) [circle, fill, inner sep=1pt] {};
				\coordinate[label=below right:$\eta$] (Z) at (0.5,-3.97);
				\node at (Z) [circle, fill, inner sep=1pt] {};
				\coordinate[label=right:$q_n$] (Q) at (0.5,-2);
				\node at (Q) [circle, fill, inner sep=1pt] {};
				\draw [black, thick] (P) to (X);
				\draw [black, thick] (P) to (Q);
				\draw [blue] (-0.22,-0.373) arc (220:290:0.3);
				\node at (-0,-0.6) {$\scriptstyle\theta_n$};
				\draw [blue] (Q) to [in=50, out=150] (X);
				\draw [red] (Q) to [in=190, out=71] (T1);
				\draw [red] (Q) to [in=40, out=140] (T2);
				\draw [blue, fill, opacity=0.3] (P) circle (1.5);
				\draw [black] (P) circle (1.5);
				\node at (0,1.8) {$B(p,r_n)$};
				\node [red] at (5,2) {$\text{pr}_{q_n}B(p,r_n)$};
				\draw [thick, blue, dashed, ->] (X) to [in=170, out=-30] (Y);
				\draw [thick, blue, dashed, ->] (Q) to (Z);
			\end{tikzpicture}
		}
		\caption{Shadow of the Ball.}\label{fig3}
	\end{figure}
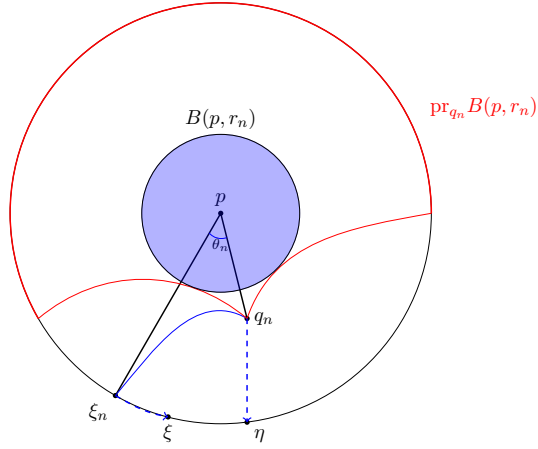

	Hence, we have
	\begin{displaymath}
		\measuredangle_p(\xi,\eta)\leq\measuredangle_p(\xi,\xi_n)+\measuredangle_p(\xi_n,q_n)+\measuredangle_p(q_n,\eta)<\frac 3n\to0.
	\end{displaymath}

	This suggests $\xi=\eta$, a contradiction.
\end{proof}

Now we can extend some important results called Shadow Lemma on manifolds with negative curvature to our setting. The first one is the Sullivan Shadow Lemma stated as the following:

\begin{theorem}[Sullivan Shadow Lemma]\label{thm_4_4}
	Let $\widetilde{M}$ be a complete simply connected and visibility manifold with no conjugate points, and $\Gamma$ be a non-elementary discrete group of $\text{Iso}(\widetilde{M})$. Suppose ${\{\mu_q\}}_{q\in\widetilde{M}}$ is the Patterson-Sullivan measure and $p\in\widetilde{M}$ is an arbitrarily chosen point, then there is some $R_0=R_0(p)>0$, such that for any $r\geq R_0$, there is $C=C(r)>1$ with the following inequality holds:
	\begin{displaymath}
		\frac 1C\leq \frac{\mu_p(\text{pr}_p(B(\alpha p,r)))}{\mathrm{e}^{-\delta_\Gamma d(p,\alpha p)}}\leq C,\qquad\forall\alpha\in\Gamma.
	\end{displaymath}
\end{theorem}

\begin{proof}
	Given $\Gamma$ being non-elementary, $\mu_p$ is not a single atom measure. Otherwise, the support of $\mu_p$, say $\xi$, is fixed under $\Gamma$, contradicting to that $\Gamma$ should be non-elementary and the support of $\mu_p$ should be $L(\Gamma)$, which is an infinity set.

	Suppose $\mu_p$ has positive measure on some points, choose $\eta\in\widetilde{M}(\infty)$ with the maximal measure of $\mu_p$ among these points and let $\lambda=\frac 12(\mu_p(\eta)+\mu_p(\widetilde{M}(\infty)))$. Clearly $\mu_p(\eta)<\lambda<\mu_p(\widetilde{M}(\infty))$.

	\textbf{Claim}: there is an $\epsilon=\epsilon(p)>0$ depending on $p$ only, such that for any $\epsilon$-ball $B(\epsilon)$ under $d_p$ on the ideal boundary, we have
	\begin{displaymath}
		\mu_p(B(\epsilon))\leq\lambda<\mu_p(\widetilde{M}(\infty)).
	\end{displaymath}

	Assume the contrary, if we can not find such $\epsilon$, then there exist ${\{\epsilon_n\}}_{n=1}^\infty\subset\mathbb{R}^+$ where $\epsilon<\frac 1n$ with some $\epsilon_n$-ball $B(\epsilon_n)$ centered at $\xi_n\in\widetilde{M}(\infty)$ whose measure $\mu_p(B(\epsilon_n))$ is greater than $\lambda$.

	By the compactness of the ideal boundary, we can assume $\xi_n\to\xi$ is a converging sequence. As $\epsilon_n\to0$, we have that
	\begin{displaymath}
		\mu_p(\xi)\geq\lambda>\mu_p(\eta),
	\end{displaymath}
	contradicting to the choice of $\eta$.

	Therefore, we can choose this $\epsilon$ as in the claim. By Proposition~\ref{prop_4_3}, we can find $R=R(p,\epsilon)$, such that for any $r\geq R$ and $d(q,p)>r$, we have
	\begin{displaymath}
		\widetilde{M}(\infty)-\text{pr}_q(B(p,r))\subset B(C_{p,q}(+\infty),\epsilon).
	\end{displaymath}

	Because $\Gamma$ is discrete, except for at most finitely many elements in $\Gamma$, we have
	\begin{displaymath}
		d(p,\alpha p)>R,~\quad \alpha\in\Gamma.
	\end{displaymath}

	For a given element $\alpha\in\Gamma$, choose $r$ between $R$ and $d(p,\alpha p)$ and let $\delta_r=\mu_p(\text{pr}_{\alpha p}B(p,r))$, we can compute that
	\begin{displaymath}
		\begin{aligned}
			\delta_r & =\mu_p(\widetilde{M}(\infty))-(\mu_p(\widetilde{M}(\infty))-\delta_r)                   \\
			         & =\mu_p(\widetilde{M}(\infty))-\mu_p(\widetilde{M}(\infty)-\text{pr}_{\alpha p}(B(p,r))) \\
			         & \geq \mu_p(\widetilde{M}(\infty))-\lambda                                               \\
			         & =\frac 12(\mu_p(\widetilde{M})(\infty)-\mu_p(\eta))                                     \\
			         & >0.
		\end{aligned}
	\end{displaymath}

	We also have for the measure in the PS measure,
	\begin{displaymath}
		\begin{aligned}
			\mu_p(\text{pr}_p B(\alpha p,r)) & =\mu_{\alpha^{-1}p}(\text{pr}_{\alpha^{-1}p}B(p,r))                                                                   \\
			                                 & =\int_{\text{pr}_{\alpha^{-1}p}B(p,r)}\,\mathrm{d}\mu_{\alpha^{-1}p}(\xi)                                             \\
			                                 & =\int_{\text{pr}_{\alpha^{-1}p}B(p,r)}\mathrm{e}^{-\delta_\Gamma\beta_{\xi}(\alpha^{-1}p,p)}\,\mathrm{d}\mu_{p}(\xi). \\
		\end{aligned}
	\end{displaymath}

	By Lemma~\ref{lem_4_2}, as $\alpha$ is an isometry, we have
	\begin{displaymath}
		d(p,\alpha p)-2r\leq\beta_{\xi}(\alpha^{-1}p,p)\leq d(p,\alpha p).
	\end{displaymath}

	By replacing the $\beta_{\xi}$ part in the integral, we can conclude that
	\begin{displaymath}
		\mu_p(\text{pr}_p B(\alpha p,r))\mathrm{e}^{-\delta_\Gamma d(p,\alpha p)}\leq \mu_p(\text{pr}_p B(\alpha p,r))\leq\mu_p(\widetilde{M}(\infty))\mathrm{e}^{2\delta_\Gamma r-\delta_\Gamma d(p,\alpha p)}.
	\end{displaymath}

	Let $C=\max\{\frac{2}{\mu_p(\widetilde{M}(\infty))-\mu_p(\eta)},\mathrm{e}^{2\delta_\Gamma r}\mu_p(\widetilde{M}(\infty))\}$ be the constant we need, and this gives the required inequality.
\end{proof}

Similarly, we can prove the following lemma for compact subset of $\widetilde{M}$ which was originally proven on manifolds with negative curvature:

\begin{theorem}[Mohsen Shadow Lemma\cite{PPS}]\label{thm_4_5}
	Let $\widetilde{M}$ be a complete simply connected and visibility manifold with no conjugate points and $\Gamma$ is a non-elementary discrete group of $\text{Iso}(\widetilde{M})$. Suppose ${\{\mu_q\}}_{q\in\widetilde{M}}$ is the Patterson-Sullivan measure and $K\subset\widetilde{M}$ is a compact set.
	For $R>0$ large enough, there is $C>0$, such that for any $\alpha\in\Gamma$, and any pair of points $p,q\in K$, we have
	\begin{displaymath}
		\frac 1C\leq \frac{\mu_p(\text{pr}_p(B(\alpha q,R)))}{\mathrm{e}^{-\delta_\Gamma d(p,\alpha q)}}\leq C.
	\end{displaymath}
\end{theorem}

\begin{proof}
	For the numerator above, we have
	\begin{equation}\label{equ_4_1}
		\mu_p(\text{pr}_p(B(\alpha q,R)))=\int_{\text{pr}_p B(\alpha q,R)}\mathrm{e}^{-\delta_\Gamma\beta_\xi(p,\alpha q)}\,\mathrm{d}\mu_{\alpha q}(\xi).
	\end{equation}

	$\Gamma$ is discrete, thus except for at most finitely many elements, we can assume $\alpha\in\Gamma$ satisfies $d(p,\alpha q)>R$. By Lemma~\ref{lem_4_2}, for any $\xi\in\text{pr}_p B(\alpha q,R)$, we have
	\begin{displaymath}
		d(p,\alpha q)-2R\leq\beta_{\xi}(p,\alpha q)\leq d(p,\alpha q).
	\end{displaymath}

	By replacing the $\beta_{\xi}$ term in the Equation~\ref{equ_4_1}, we can conclude
	\begin{displaymath}
		\mathrm{e}^{-\delta_{\Gamma}d(p,\alpha q)}\mu_{\alpha q}(\text{pr}_p B(\alpha q,R))\leq\mu_p(\text{pr}_p B(\alpha q,R))\leq\mathrm{e}^{2\delta_\Gamma R-\delta_\Gamma d(p,\alpha q)}\mu_{\alpha q}(\text{pr}_p B(\alpha q,R)).
	\end{displaymath}

	Divided by the denominator, we only need to prove the following statement.

	\textbf{Claim}: For $R>0$ large enough, there is $C_1>0$, such that for any $\alpha\in\Gamma$ and any pairs of point $p,q\in K$, we have
	\begin{displaymath}
		\frac{1}{C_1}\leq\mu_{\alpha q}(\text{pr}_p B(\alpha q, R))\leq C_1,
	\end{displaymath}

	where $C=C_1\mathrm{e}^{2\delta_\Gamma R}$ is given in this claim.

	To prove this claim, let
	\begin{displaymath}
		A=\sup_{\substack{p,q\in K\\ \xi\in\widetilde{M}(\infty)}}\delta_\Gamma |\beta_{\xi}(p,q)|\leq\delta_\Gamma\cdot\text{diam}{K}<\infty.
	\end{displaymath}

	We have for any point $p\in K$, the upper bound in the claim holds as
	\begin{displaymath}
		\begin{aligned}
			\mu_{\alpha q}(\text{pr}_p B(\alpha q,R)) & \leq\mu_{\alpha q}(\widetilde{M}(\infty))                                                        \\
			                                          & =\mu_{q}(\alpha^{-1}\widetilde{M}(\infty))                                                       \\
			                                          & =\mu_{q}(\widetilde{M}(\infty))                                                                  \\
			                                          & =\int_{\widetilde{M}(\infty)}\mathrm{e}^{-\delta_{\Gamma}\beta_{\xi}(p,q)}\,\mathrm{d}\mu_p(\xi) \\
			                                          & \leq\mathrm{e}^{A}\mu_p(\widetilde{M}(\infty)).
		\end{aligned}
	\end{displaymath}

	For the lower bound, assume the contrary. We can find sequences of points $p_n$ and $q_n$ in $K$, with $\{\alpha_n\}\subset\Gamma$ and $R_n\to\infty$ correspondingly, such that
	\begin{displaymath}
		\lim_{n\to\infty} \mu_{\alpha_n q_n}(\text{pr}_{p_n} B(\alpha_n q_n, R_n))=0.
	\end{displaymath}

	By the compactness of $K$, passing to sub-sequences if needed, we can assume $p_n\to p$ and $q_n\to q$ are converging sequences. Because $\Gamma$ is discrete, under the cone topology, we can find a point $\xi$ on the ideal boundary such that,
	\begin{displaymath}
		\lim_{n\to\infty}\alpha_n^{-1}p=\xi\in\widetilde{M}(\infty).
	\end{displaymath}

	By Proposition~\ref{prop_4_3}, for any relatively compact open subset $V\subset\widetilde{M}(\infty)-{\xi}$ and $n$ sufficiently large, we have
	\begin{displaymath}
		V\subset \text{pr}_{\alpha_n^{-1}p_n}B(q_n,R_n).
	\end{displaymath}

	We can compute that
	\begin{equation}
		\begin{aligned}
			\mu_q(V) & =\mu_{\alpha_n q}(\alpha_n V)                                                                                                                  \\
			         & \leq \mu_{\alpha_n q}(\alpha_n \text{pr}_{\alpha_n^{-1}p_n} B(q_n,R_n))                                                                        \\
			         & =\mu_{\alpha_n q}(\text{pr}_{p_n} B(\alpha_n q_n, R_n))                                                                                        \\
			         & =\int_{\text{pr}_{p_n}B(\alpha_n q_n, R_n)}\mathrm{e}^{-\delta_\Gamma\beta_{\xi}(\alpha_n q, \alpha_n q_n)}\,\mathrm{d}\mu_{\alpha_n q_n}(\xi) \\
			         & =\int_{\text{pr}_{p_n}B(\alpha_n q_n, R_n)}\mathrm{e}^{-\delta_\Gamma\beta_{\alpha_n^{-1}\xi}(q, q_n)}\,\mathrm{d}\mu_{\alpha_n q_n}(\xi)      \\
			         & \leq\mathrm{e}^{A}\mu_{\alpha_n q_n}(\text{pr}_{p_n}B(\alpha_n q_n, R_n))
		\end{aligned}
	\end{equation}

	By assumption, the last line converges to $0$, this suggests $\mu_q(V)=0$ for arbitrary $V$ relatively compact. Therefore, $\text{supp}(\mu_q)=\{\xi\}$, $\mu_q$ is an atomic measure, contradicting to the condition that $\Gamma$ is non-elementary.

	Therefore, there is a positive lower bound. By choosing a suitable constant $C_1$, we have proven this lemma.
\end{proof}

\section{\bf Conical Limit Set}\label{sec5}
\begin{definition}
	We call a limit point $\xi\in\widetilde{M}(\infty)$ a conical limit point, if for any point $p\in\widetilde{M}$, there is $D=D(p)>0$ depending on $p$ only and $\{\alpha_n\}\subset\Gamma$, such that $d(\alpha_n p, c_{p,\xi})\leq D$ and $\alpha_n p\to\xi$. Denote the set of all conical limit points by $L_c(\Gamma)$. 
\end{definition}

As a remark, it is straightforward to check that given $\xi$, the choice of $p$ does not affect if it is a conical limit point or not (cf. Proposition~\ref{prop_2_7} (3)). That is to say, if we can find one point $p_0$ satisfying the above condition, then for any $p\in\widetilde{M}$, there is some $D=D(p)$ such that the condition holds with the same sequence of $\alpha_n$. Therefore, to check if a limit point $\xi$ is a conical limit point, we only need to check this at one point.

For conical limit points, we have the following proposition:

\begin{proposition}\label{prop_5_2}
	Fix a point $p_0\in\widetilde{M}$, if $\xi\in L_c(\Gamma)$, we can find $\{\alpha_n\}$, such that
	\begin{displaymath}
		\lim_{n\to\infty}\beta_{\xi}(p_0,\alpha_n p_0)=+\infty,
	\end{displaymath}
	but ${\{d(p_0,\alpha_n p_0)-\beta_{\xi}(p_0,\alpha_n p_0)\}}_{n=1}^\infty$ is bounded.
\end{proposition}

\begin{proof}
	$\xi$ is a conical limit point, we can find a number $D=D(p_0)$ and $\{\alpha_n\}\subset\Gamma$, such that $\alpha_n p_0\to\xi$ under the cone topology and $d(\alpha_n p_0, c_{p_0,\xi})<D$ is bounded.

	Therefore, $\xi\in\text{pr}_{p_0}B(\alpha_n p_0, D)$ for all $n\in\mathbb{N}$. By Lemma~\ref{lem_4_2}, we can conclude
	\begin{displaymath}
		d(p_0,\alpha_n p_0)-2D\leq \beta_\xi(p_0,\alpha_n p_0)\leq d(p_0,\alpha_n p_0),
	\end{displaymath}
    as required.
\end{proof}

The following theorem gives an equivalent condition for a point to be a conical limit point.
\begin{theorem}\label{thm_5_3}
	Let $\widetilde{M}$ be a complete simply connected, uniform visibility manifold with no conjugate points. Then $\xi\in L_c(\Gamma)$ if and only if there exists a sequence of distinct isometries ${\{\alpha_n\}}_{n=1}^\infty\subset\Gamma$, such that for any $\eta\in\widetilde{M}(\infty)-\{\xi\}$, the pairs ${\{(\alpha_n\xi, \alpha_n\eta)\}}_{n=1}^\infty$ are contained in a compact subset of $\widetilde{M}^2(\infty)=\widetilde{M}(\infty)\times\widetilde{M}(\infty)-\{(\xi,\xi)\mid\xi\in\widetilde{M}(\infty)\}$.
\end{theorem}

\begin{proof}
	Pick a compact subset $A\subset\widetilde{M}^2(\infty)$ and a point $p_0\in\widetilde{M}$. Under the cone topology, the angle between two points on the ideal boundary measured at the point $p_0$ is a continuous function. Let $\delta=\min\limits_{(\eta_1,\eta_2)\in A}\measuredangle_{p_0}(\eta_1,\eta_2)$. We have $\delta>0$ as $A$ is compact and does not intersect with the diagonal.

	For any pairs of point $(\eta_1,\eta_2)\in A$ and all connecting geodesics $c_{\eta_1,\eta_2}$, by the visibility of $\widetilde{M}$, we can find some $R=R(p_0,\delta)>0$, such that
	\begin{displaymath}
		d(p_0,c_{\eta_1,\eta_2})\leq R.
	\end{displaymath}

	Therefore, if ${\{(\alpha_n\xi, \alpha_n\eta)\}}_{n=1}^\infty\subset A$, we have $d(p_0,c_{\alpha_n\xi,\alpha_n\eta})=d(\alpha_n^{-1}p_0, c_{\xi,\eta})$ is bounded for all $n$.

	For the `only if' part, if $\xi\in L_{c}(\Gamma)$, we can find $D=D(p_0)$ and $\{\alpha_n\}\subset\Gamma$, such that $\alpha_n^{-1}p_0\to\xi$ and $d(\alpha_n^{-1}p_0, c_{p_0,\xi})\leq D$ bounded. Let $q_0$ be the intersection point of the horosphere $H_{\xi}(p_0)$ and any connecting geodesic $c_{\xi,\eta}$ (choose any if the connecting geodesic is not unique).
	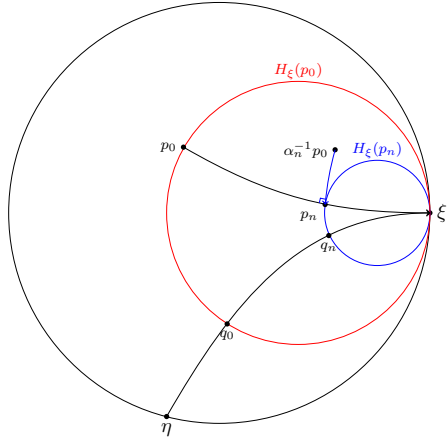
\begin{figure}[htbp]
		\centering
		\scalebox{0.7}{
			\begin{tikzpicture}
				\draw (0,0) circle (4);
				\coordinate [label=right:$\xi$] (X) at (4,0);
				\node at (X) [circle, fill, inner sep=1pt] {};
				\draw (3,0) [blue] circle (1);
				\node at (3,1.2) [blue] {$\scriptstyle H_{\xi}(p_n)$};
				\draw (1.5,0) [red] circle (2.5);
				\node at (1.5,2.7) [red] {$\scriptstyle H_{\xi}(p_0)$};
				\coordinate [label=below:$\eta$] (Y) at (-1,-3.87);
				\node at (Y) [circle, fill, inner sep=1pt] {};
				\draw [->] (Y) to [in=180, out=60] (X);
				\coordinate [label=left:$\scriptstyle p_0$] (A) at (-0.68,1.25);
				\node at (A) [circle, fill, inner sep=1pt] {};
				\draw [->] (A) to [in=180, out=-30] (X);

				\coordinate [label=below:$\scriptstyle q_0$] (Q0) at (0.15,-2.11);
				\node at (Q0) [circle, fill, inner sep=1pt] {};
				\coordinate [label=below:$\scriptstyle q_n$] (Qn) at (2.08,-0.43);
				\node at (Qn) [circle, fill, inner sep=1pt] {};
				\coordinate [label=below left:$\scriptstyle p_n$] (Pn) at (2.01,0.16);
				\node at (Pn) [circle, fill, inner sep=1pt] {};
				\coordinate [label=left:$\scriptstyle \alpha_n^{-1}p_0$] (P) at (2.2,1.2);
				\node at (P) [circle, fill, inner sep=1pt] {};
				\draw [->, blue](P) to [in=84, out=255](Pn);
				\draw [blue] (1.9,0.19) to (1.92,0.28) to (2.02,0.26);
			\end{tikzpicture}
		}
        \caption{Control the distance to $c_{p_0,\eta}$.}\label{fig4}
	\end{figure}

	As in Figure~\ref{fig4}, denote the projection of $\alpha_n^{-1}p_0$ onto $c_{p_0,\xi}$ by $p_n$, and the intersection of $H_{\xi}(p_n)$ and $c_{\xi,\eta}$ by $q_n$. By Proposition~\ref{prop_2_7}~(1), we have
	\begin{displaymath}
        d(\alpha_n^{-1}p_0, c_{\xi,\eta})\leq d(\alpha_n^{-1}p_0,p_n)+d(p_n,q_n)\leq D+2R(\frac{\pi}{2})+3d(p_0,q_0).
	\end{displaymath}

	Use the equivalent condition above, we have ${\{(\alpha_n\xi,\alpha_n\eta)\}}_{n=1}^\infty$ is contained in some compact subset of $\widetilde{M}^2(\infty)$.

	For the `if' part, we can find a sequence of distinct isometries $\{\alpha_n\}$, such that for any $\eta\in\widetilde{M}-\{\xi\}$, the pairs $\{(\alpha_n\xi,\alpha_n\eta)\}$ is contained in some compact subset of $\widetilde{M}^2(\infty)$. The above equivalent condition suggests that $d(p_0,c_{\alpha_n\xi,\alpha_n\eta})$ is bounded. As $\Gamma$ is discrete and the choice of $\eta$ is arbitrary, we have that
	\begin{displaymath}
		\lim_{n\to\infty}\alpha_n p_0=\xi.
	\end{displaymath}

	Together with the boundness, we get $\xi$ is a conical limit point.
\end{proof}

For the rest of this section, we will study the measures regarding the conical limit points set.

\begin{theorem}\label{thm_5_4}
	Suppose $\xi\in L_c(\Gamma)$ and $\mu={\{\mu_q\}}_{q\in\widetilde{M}}$ is the Patterson-Sullivan measure with $\delta_\Gamma>0$. Then $\xi$ cannot be an atom of $\mu$.
\end{theorem}

\begin{proof}
	Assume the contrary, suppose we have $\mu_q(\xi)>0$ for some $q\in\widetilde{M}$. As $\xi$ is a conical limit point, we can find $D=D(q)$ and $\{\alpha_n\}\subset\Gamma$, such that $\alpha_n q\to\xi$ and $d(\alpha_n p, c_{q,\xi})<D$. Thus, $\xi\in\text{pr}_q B(\alpha_n q,D)$.

	By Lemma~\ref{lem_4_2}, we have
	\begin{displaymath}
		d(q,\alpha_n q)-2D\leq\beta_{\xi}(q,\alpha_n q)\leq d(q,\alpha_n q).
	\end{displaymath}

	Since $\alpha_n q\to\xi$, we have $d(q,\alpha_n q)\to\infty$, implying $\beta_{\xi}(q,\alpha_n q)\to\infty$.

	\textbf{Claim:}~By passing to a sub-sequence if needed, the set ${\{\alpha_n^{-1}\xi\}}_{n=1}^\infty$ is a set of distinct points in $\widetilde{M}(\infty)$.

	If this claim holds, we can derive that
	\begin{displaymath}
		\begin{aligned}
			\sum_{n=1}^\infty\mathrm{e}^{-\delta_\Gamma\beta_{\xi}(\alpha_n q,q)} & =\sum_{n=1}^\infty \frac{\mu_{\alpha_n q}(\xi)}{\mu_q(\xi)}     \\
			                                                                      & =\frac{1}{\mu_q(\xi)}\sum_{n=1}^\infty\mu_{q}(\alpha_n^{-1}\xi) \\
			                                                                      & \leq\frac{1}{\mu_q(\xi)}\mu_q(\widetilde{M}(\infty))<\infty.
		\end{aligned}
	\end{displaymath}

	Which contradicts to the fact $\beta_\xi(\alpha_n q,q)=-\beta_\xi(q,\alpha_n q)\to -\infty$.

	Now we are going to verify this claim. Assume the contrary, there is a point $\eta\in\widetilde{M}(\infty)$ such that $\alpha_n^{-1}\xi=\eta$ for all $n\in\mathbb{N}$, with the distance $d(q,\alpha_n q)$ increasing with $n$. This can be done as $\Gamma$ is discrete. Therefore, we can conclude that
	\begin{displaymath}
		\begin{aligned}
			1 & =\frac{\mu_{\alpha_n^{-1}q}(\alpha_n^{-1}\xi)}{\mu_q(\xi)}                                       \\
			  & =\frac{\mu_{\alpha_n^{-1}q}(\eta)}{\mu_q(\xi)}                                                   \\
			  & =\frac{\mu_{\alpha_n^{-1}q}(\alpha_1^{-1}\xi)}{\mu_q(\xi)}                                       \\
			  & =\frac{\mu_{\alpha_1\alpha_n^{-1}q}(\xi)}{\mu_q(\xi)}                                            \\
			  & =\mathrm{e}^{-\delta_\Gamma\beta_{\xi}(\alpha_1\alpha_n^{-1}q,q)},\qquad \forall n\in\mathbb{N}.
		\end{aligned}
	\end{displaymath}

	As $\delta_\Gamma\neq0$, we have for all $n\in\mathbb{N}$,
	\begin{displaymath}
		\beta_\xi(\alpha_1\alpha_n^{-1}q,q)=\beta_{\alpha_1^{-1}\xi}(\alpha_n^{-1}q,\alpha_1^{-1}q)=0.
	\end{displaymath}

	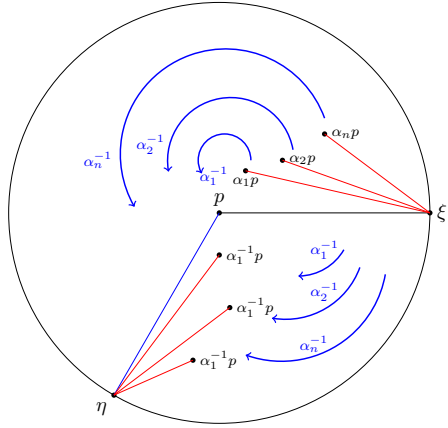
\begin{figure}[htbp]
		\centering
		\scalebox{0.7}{
			\begin{tikzpicture}
				\coordinate [label=above:$p$] (P) at (0,0);
				\node at (P) [circle, fill, inner sep=1pt] {};
				\draw (P) circle (4);
				\coordinate [label=right:$\xi$] (X) at (4,0);
				\node at (X) [circle, fill, inner sep=1pt] {};
				\coordinate [label=below left:$\eta$] (Y) at (-2,-3.464);
				\node at (Y) [circle, fill, inner sep=1pt] {};
				\draw (P) to (X);
				\draw [blue](P) to (Y);
				\coordinate[label=below:$\scriptstyle\alpha_1 p$] (A1) at (0.5,0.8);
				\node at (A1) [circle, fill, inner sep=1pt] {};
				\draw [blue, ->, thick] (0.6,1) arc (0:210:0.5);
				\node at (-0.1,0.7) [blue] {$\scriptstyle\alpha_1^{-1}$};
				\coordinate[label=right:$\scriptstyle\alpha_2 p$] (A2) at (1.2,1);
				\node at (A2) [circle, fill, inner sep=1pt] {};
				\draw [blue, ->, thick] (1.4,1.2) arc (10:190:1.2);
				\node at (-1.3,1.3) [blue]{$\scriptstyle\alpha_2^{-1}$};
				\coordinate[label=right:$\scriptstyle\alpha_n p$] (A3) at (2,1.5);
				\node at (A3) [circle, fill, inner sep=1pt] {};
				\draw [blue, ->, thick] (2,1.8) arc (20:210:2);
				\node at (-2.3,1) [blue]{$\scriptstyle\alpha_n^{-1}$};
				\draw[red] (A1) to (X);
				\draw[red] (A2) to (X);
				\draw[red] (A3) to (X);

				\draw[blue, <-, thick] (1.5,-1.2) arc (-90:-30:1);
				\node at (2,-0.7) [blue]{$\scriptstyle\alpha_1^{-1}$};
				\draw[blue, <-, thick] (1,-2) arc (-100:-20:1.5);
				\node at (2,-1.5) [blue]{$\scriptstyle\alpha_2^{-1}$};
				\draw[blue, <-, thick] (0.5,-2.7) arc (-110:-10:2);
				\node at (1.8,-2.5) [blue] {$\scriptstyle\alpha_n^{-1}$};
				\coordinate [label=right:$\scriptstyle\alpha_1^{-1}p$] (B1) at (0.2,-1.8);
				\node at (B1) [circle, fill, inner sep=1pt] {};
				\coordinate [label=right:$\scriptstyle\alpha_1^{-1}p$] (B2) at (0,-0.8);
				\node at (B2) [circle, fill, inner sep=1pt] {};
				\coordinate [label=right:$\scriptstyle\alpha_1^{-1}p$] (B3) at (-0.5,-2.8);
				\node at (B3) [circle, fill, inner sep=1pt] {};
				\draw [red] (B1) to (Y);
				\draw [red] (B2) to (Y);
				\draw [red] (B3) to (Y);
			\end{tikzpicture}
		}
		\caption{Under the Isometry of $\alpha_n$'s.}\label{fig5}
	\end{figure}
	Thus, $\alpha_n^{-1}q$ lies on the horosphere $H_\eta(\alpha_1^{-1}q)$, and $d(\alpha_n^{-1}q, c_{q,\eta})<D$. But $\Gamma$ is discrete, this is not possible. The claim holds, as well as the theorem.

\end{proof}

\begin{theorem}\label{thm_5_5}
	Let $\widetilde{M}$ be a complete simply connected, uniform visibility manifold with no conjugate points, and $\Gamma$ be a non-elementary discrete group of $\text{Iso}(\widetilde{M})$. Let $\mu={\{\mu_q\}}_{q\in\widetilde{M}}$ be the Patterson-Sullivan measure. For a point $p\in\widetilde{M}$, if the corresponding Poincar\'e series converges at the critical exponent:
	\begin{displaymath}
		P(\delta_\Gamma,p,p)=\sum_{\alpha\in\Gamma}\mathrm{e}^{-\delta_\Gamma d(p,\alpha p)}<\infty,
	\end{displaymath}
	then $\mu_p(L_c(\Gamma))=0$, namely the conical limit points set has measure zero. Actually, for all $q\in\widetilde{M}$, the measure $\mu_q(L_c(\Gamma))$ is zero.
\end{theorem}

\begin{proof}
	Given $\Gamma$ a discrete group, list all elements in $\Gamma$ as $\{\alpha_n\}$. We can write
	\begin{displaymath}
		L_c(\Gamma)=\bigcup_{k=1}^\infty\bigcap_{N=1}^{\infty}\bigcup_{n=N}^\infty \text{pr}_p B(\alpha_n p,k).
	\end{displaymath}

	Let $L^k_c(\Gamma)=\bigcap_{N=1}^{\infty}\bigcup_{n=N}^\infty \text{pr}_p B(\alpha_n p,k)$, it is enough to show that $\mu_p(L^k_c(\Gamma))=0$ for all $k\in\mathbb{N}$.

	By definition, we have that
	\begin{equation}
		\begin{aligned}
			\mu_p(L^k_c(\Gamma)) & =\mu_p(\bigcap_{N=1}^{\infty}\bigcup_{n=N}^\infty \text{pr}_p B(\alpha_n p,k)) \\
			                     & \leq\sum_{n=N}^\infty\mu_p(\text{pr}_p B(\alpha_n p,k))                        \\
			                     & \leq C\sum_{n=N}^\infty\mathrm{e}^{-\delta_\Gamma d(p,\alpha_n p)}.
		\end{aligned}
	\end{equation}

	The last line follows from the Sullivan Shadow Lemma~\ref{thm_4_4}, and $C$ is the constant given there. Note that this inequality hold for any $N\in\mathbb{N}$, and the Poincar\'e series converges implies the tail of the series tends to $0$. Therefore, $\mu_p(L_c^k(\Gamma))=0$ as required.
\end{proof}

Under a similar setting, we can discuss the measure of the invariant set in the conical limit point set.

\begin{theorem}\label{thm_5_6}
	Let $\widetilde{M}$ be a complete simply connected, uniform visibility manifold with no conjugate points, and $\Gamma$ be a non-elementary discrete group of $\text{Iso}(\widetilde{M})$. Let $\mu={\{\mu_q\}}_{q\in\widetilde{M}}$ be the Patterson-Sullivan measure. Suppose $A\subset L_c(\Gamma)$ is a $\Gamma$-invariant subset, and $p\in\widetilde{M}$ is an arbitrary point. Then $\mu_p(A)$ is trivial, either $\mu_p(A)=0$ or $\mu_p(A)=\mu_p(\widetilde{M}(\infty))$.
\end{theorem}

\begin{proof}
	As all the measures in the PS measure are equivalent, it is sufficient to show this is true for a given point. Fix a point $p\in\widetilde{M}$. $\Gamma$ is non-elementary implies that $\mu_p$ is not a single atom (see the discussion in the proof of the Sullivan Shadow Lemma~\ref{thm_4_4}).

	For any $A\subset L_c(\Gamma)$ an invariant set with $\mu_p(A)>0$, we want to show that it has full measure $\mu_p(\widetilde{M}(\infty))$. Pick a density point $\xi\in A$, as a conical limit point, we can find $\{\alpha_n\}\subset\Gamma$, such that $\alpha_n^{-1}(p)\to\xi$ under the cone topology, and as a density point, we have
	\begin{equation}\label{equ_5_1}
		\frac{\mu_p(\text{pr}_p B(\alpha_n^{-1}p,C)\cap A)}{\mu_p(\text{pr}_p B(\alpha_n^{-1}p,C))}\to 1.
	\end{equation}

	Here $C$ is a positive constant and can be chosen arbitrarily by Proposition~\ref{prop_2_7}~(3).

	Suppose $\lambda$ is the maximal possible point mass of $\mu_p$. We claim that for any $\epsilon>0$, there is $\delta_0=\delta_0(\lambda,\epsilon)>0$, such that the measure of any open ball in $\widetilde{M}(\infty)$, whose radius is less than $\delta_0$ under the $d_p$ metric, is no more than $\lambda+\epsilon$.

	Again we prove this claim by contradiction. Assume that there are open balls $B_n\subset\widetilde{M}(\infty)$ with radius $\delta_n$ decreasing to $0$, but $\mu_p(B_n)>\lambda+\epsilon$. By the compactness of the ideal boundary, we can assume that the center of $B_n$ is converging to a point $\xi$. This immediately suggests that $\mu_p(\xi)\geq \lambda+\epsilon>\lambda$, contradiction to the maximality of $\lambda$.

	By Proposition~\ref{prop_4_3}, we can choose some $C$, such that for any $x\in\widetilde{M}$, $d(x,p)>C$ implies $\widetilde{M}(\infty)-\text{pr}_x B(p,C)$ is contained in an open ball with radius smaller than $\delta_0$ defined above. Since $\Gamma$ is discrete, we can find some $n$ sufficiently large with $d(p,\alpha_n p)>C$. Therefore,
	\begin{equation}\label{equ_5_3}
		\begin{aligned}
			\mu_p(\text{pr}_{\alpha_n p}B(p,C)) & =\mu_p(\widetilde{M}(\infty))-\mu_p(\widetilde{M}(\infty)-\text{pr}_{\alpha_n p}B(p,C)) \\
			                                    & \geq\mu_p(\widetilde{M}(\infty))-\lambda-\epsilon.
		\end{aligned}
	\end{equation}

	Using the equivariance of the measures and the invariance of $A$ under $\Gamma$, we conclude that
	\begin{equation}\label{equ_5_4}
		\begin{aligned}
			\frac{\mu_p(\text{pr}_{\alpha_n p} B(p,C)\cap A)}{\mu_p(\text{pr}_{\alpha_n p} B(p,C))} & =\frac{\mu_{\alpha_n^{-1}p}(\text{pr}_p B(\alpha_n^{-1}p,C)\cap A)}{\mu_{\alpha_n^{-1}p}(\text{pr}_{p} B(\alpha_n^{-1}p,C))}                                                                                                                             \\
			                                                                                        & =\frac{\int_{\text{pr}_p B(\alpha_n^{-1}p,C)\cap A}\mathrm{e}^{-\delta_\Gamma\beta_{\xi}(\alpha_n^{-1}p,p)}\,\mathrm{d}\mu_p(\xi)}{\int_{\text{pr}_p B(\alpha_n^{-1}p,C)}\mathrm{e}^{-\delta_\Gamma\beta_{\xi}(\alpha_n^{-1}p,p)}\,\mathrm{d}\mu_p(\xi)} \\
		\end{aligned}
	\end{equation}

	By Lemma~\ref{lem_4_2}, we have
	\begin{displaymath}
		d(p,\alpha_n p)-2C\leq\beta_{\xi}(p,\alpha_n^{-1}p)\leq d(p,\alpha_n p).
	\end{displaymath}

	By replacing the $\beta_{\xi}$ part in Equation~\ref{equ_5_4}, we get
	\begin{displaymath}
		\frac{\mu_p(\text{pr}_{\alpha_n p} B(p,C)\cap A)}{\mu_p(\text{pr}_{\alpha_n p} B(p,C))}\geq 1-\mathrm{e}^{2C\delta_\Gamma}\frac{\mu_p(\text{pr}_p B(\alpha_n^{-1}p,C)\cap A^c)}{\mu_p(\text{pr}_p B(\alpha_n^{-1}p,C))}.
	\end{displaymath}

	For any given $\epsilon$, $\xi$ is a density point, we can choose $n$ sufficiently large, such that
	\begin{displaymath}
		\frac{\mu_p(\text{pr}_p B(\alpha_n^{-1}p,C)\cap A)}{\mu_p(\text{pr}_p B(\alpha_n^{-1}p,C))}\geq 1-\mathrm{e}^{-2C\delta_\Gamma}\epsilon.
	\end{displaymath}

	Combining the last two inequalities, we can conclude that
	\begin{displaymath}
		\mu_p(\text{pr}_{\alpha_n p}B(p,C)\cap A)\geq (1-\epsilon)\mu_p(\text{pr}_{\alpha_n p}B(p,C)).
	\end{displaymath}

	Therefore, we can see that $A$ takes almost the full measure both in $\text{pr}_p B(\alpha_n^{-1}p, C)$ and $\text{pr}_{\alpha_n p}B(p,C)$, thus
	\begin{displaymath}
		\mu_p(A)\geq\mu_p(\text{pr}_{\alpha_n p}(B(p,C))\cap A)\geq(1-\epsilon)(\mu_p(\widetilde{M}(\infty))-\lambda-\epsilon).
	\end{displaymath}

	Let $\epsilon\to 0^+$, we have
	\begin{equation}\label{equ_5_5}
		\mu_p(A)\geq\mu_p(\widetilde{M}(\infty))-\lambda.
	\end{equation}

	To show $\mu_p(A)$ has the full measure, we only need to discuss on the maximal possible point mass $\lambda$ of $\mu_p$. If $\lambda=0$, we are done. Otherwise, if $\lambda>0$, by Equation~\ref{equ_5_5}, the point with maximal point mass must be contained in $A\subset L_c(\Gamma)$. But Theorem~\ref{thm_5_4} tells that there is no atom in the conical limit points set. Hence, in all cases we have that $\mu_p(A)=\mu_p(\widetilde{M}(\infty))$.
\end{proof}

\begin{remark}
	If we choose the invariant set $A=L_c(\Gamma)$, we can conclude $L_c(\Gamma)$ either has zero measure or has full measure under any $\mu_p$.
\end{remark}

Lastly, combine Theorem~\ref{thm_5_6} and Proposition~\ref{prop_3_6}, we have the following proposition.

\begin{proposition}\label{prop_5_7}
    Let $\mu={\{\mu_q\}}_{q\in\widetilde{M}}$ be the Patterson-Sullivan measure on $\widetilde{M}(\infty)$. If there exist a point $p$ (and hence for all points in $\widetilde{M}$), such that $\mu_p(L_c(\Gamma))>0$, then the action of $\Gamma$ on the ideal boundary is ergodic with respect to $\mu_p$. Furthermore, the PS measure is unique up to a positive multiple.
\end{proposition}

\begin{proof}
	$L_c(\Gamma)$ is a $\Gamma$-invariant set with $\mu_p(L_c(\Gamma))>0$, it has full measure.

	Given a $\Gamma$-invariant set $A$ on the ideal boundary. Consider $E=A\cap L_c(\Gamma)$, $E$ is again a $\Gamma$-invariant set, which has either measure $0$ or the full measure. If $\mu_p(E)=0$, we have $\mu_p(A)=\mu_p(A\cap L_c(\Gamma))=0$. If $\mu_p(E)=\mu_p(\widetilde{M}(\infty))$, then $A\supset E$ has full measure.
\end{proof}

\section{\bf Dynamics of Geodesic Flows and the HTS Dichotomy}\label{sec6}
We abuse the notation a little and let $\phi_t$ denote the geodesic flow on both $T^1 M$ and $T^1\widetilde{M}$:
\begin{displaymath}
	\phi_t(\bm{v})=c'_{\bm{v}}(t),\quad \bm{v}\in T^1M\text{~or}~T^1\widetilde{M},~t\in\mathbb{R}.
\end{displaymath}

And let $\pi$ denote the standard projection from the unit tangent bundles to the base manifolds for both $T^1M$ and $T^1\widetilde{M}$.
\begin{definition}
	Let $\phi_t$ be the geodesic flow on $T^1M$.
	\begin{itemize}
		\item We say $\bm{v}\in T^1M$ is a \emph{conservative} point of the geodesic flow, if there is a compact subset $A\subset T^1M$ with a sequence of numbers $t_n\to+\infty$, such that $\phi_{t_n}(\bm{v})\in A$ for all $n$. Let $M_C$ be the set of all conservative points.
		\item We call $\bm{v}\in T^1M$ is a \emph{dissipative} point, if for any compact subset $A\subset T^1M$, we can find some $t_A$, such that $\phi_t(\bm{v})\notin A$ for all $t>t_A$. Let $M_D$ be the set of all dissipative points.
		\item $T^1M=M_C\sqcup M_D$. For any invariant measure $m$ with respect to $\phi_t$, we call the geodesic flow is \emph{conservative}, if $m(M_D)=0$. Similarly, we call it \emph{completely dissipative} if $m(M_C)=0$.
	\end{itemize}
\end{definition}

Now we are trying to build an intrinsic relation between the conservative point and the conical limit point.

\begin{proposition}\label{prop_6_2}
	$\bm{v}\in T^1M$ is a conservative point for $\phi_t$, if and only if $c_{\bm{\widetilde{v}}}(+\infty)\in L_c(\Gamma)$ is a conical limit point for some (hence any) lift $\bm{\widetilde{v}}\in T^1\widetilde{M}$.
\end{proposition}

\begin{proof}
	Suppose $\bm{v}\in T^1M$ is a conservative point of $\phi_t$, by definition, we can find some compact subset $A$ and $t_n\to+\infty$, such that $\phi_{t_n}(\bm{v})\in A$.

	Pick a lift $\bm{\widetilde{v}}$ of $\bm{v}$, and $\widetilde{A}$ is a lift of $A$. We can find a sequence of $\{\alpha_n\}$ such that
	\begin{displaymath}
		\phi_{t_n}(\bm{\widetilde{v}})\in\alpha_n(\widetilde{A}).
	\end{displaymath}

	Project to $\widetilde{M}$, we have
	\begin{equation}\label{equ_6_1}
		c_{\bm{\widetilde{v}}}(t_n)\in\alpha_n(\pi(\widetilde{A})).
	\end{equation}

	The set $A$ is compact, thus bounded. Let $K=\text{diam}(\pi(A))<+\infty$. $\Gamma$ being an isometry subgroup implies that
	\begin{equation}\label{equ_6_2}
		\text{diam}(\alpha_n\pi(\widetilde{A}))=\text{diam}(\pi(A))=K.
	\end{equation}

	Equation~\ref{equ_6_1} and~\ref{equ_6_2} together imply that $\alpha_n p\to c_{\bm{\widetilde{v}}}(+\infty)$ and $d(\alpha_n p, c_{\bm{\widetilde{v}}})\leq K$. By definition, $c_{\bm{\widetilde{v}}}(+\infty)\in L_c(\Gamma)$.

	For the other direction, the argument is similar.
\end{proof}

Next, we will construct the invariant measure for the geodesic flow using the PS measure. This invariant measure was first constructed for manifolds with constant negative curvature by Sullivan~\cite{Su}, on manifolds with pinched negative curvature by Kaimanovich~\cite{Ka1} and later generalized to non-positive curvature case by Knieper~\cite{Kn1}. On manifolds with no focal points, we studied this measure in~\cite{LWW}. It plays an important role in study the ergodic theory of geodesic flows.

Pick a point $p\in\widetilde{M}$, we can define a measure on $\widetilde{M}^2(\infty)$ by
\begin{displaymath}
	\mathrm{d}\widetilde{\mu}(\xi,\eta)=\mathrm{e}^{\delta_\Gamma\beta_p(\xi,\eta)}\mathrm{d}\mu_p(\xi)\mathrm{d}\mu_p(\eta).
\end{displaymath}

Easy to check that this definition does not depend on the choice of $p$, and it is $\Gamma$-invariant by the equivariance of the PS measure. Furthermore, this measure induces a $\phi_t$ and $\Gamma$-invariant measure $\widetilde{m}$ on $T^1\widetilde{M}$ by
\begin{displaymath}
    \widetilde{m}(A)=\int_{\widetilde{M}^2(\infty)}\text{Length}(c_{\xi,\eta}\cap \pi(A))\,\mathrm{d}\widetilde{\mu}(\xi,\eta)
\end{displaymath}
for any Borel set $A\subset T^1\widetilde{M}$. Here $c_{\xi,\eta}$ is the unique geodesic connecting $\xi$ and $\eta$, and $\pi:T^1 \widetilde{M}\to\widetilde{M}$ is the standard projection.

By definition, this measure $\widetilde{m}$ is a $\phi_t$ and $\Gamma$ invariant measure on $T^1\widetilde{M}$, which induces a $\phi_t$ invariant measure on $T^1M$ by the standard projection map. Let $m$ denote this measure on $T^1M$, which is known as the Bowen-Margulis-Sullivan (BMS) measure. A natural question regarding this measure is that if the geodesic flow $\phi_t$ is conservative with respect to it. Actually, we get a dichotomy to answer this question as shown in the following theorem.

\begin{theorem}\label{thm_6_3}
	Let $M$ be a complete, regular and visibility manifold with no conjugate points, and $\widetilde{M}$ is the universal cover. Let $m$ denote the BMS measure defined above. We have the following dichotomy:

	If $\mu_p(L_c(\Gamma))=\mu_p(\widetilde{M}(\infty))$, then the geodesic flow $\phi_t$ is conservative with respect to $m$. Otherwise, it is completely dissipative with respect to $m$.
\end{theorem}

\begin{proof}
	The Theorem~\ref{thm_5_6} shows that for $\mu_p$ in the PS measure, $\mu_p(L_c(\Gamma))$ is either $0$ or full. By Proposition~\ref{prop_6_2} and $m$ being a product measure, we know $\mu_p(L_c(\Gamma))=\mu_p(\widetilde{M}(\infty))$ implies that $M_C$ has full measure in $T^1M$, thus $\phi_t$ is conservative.

	Similar argument works for the complete dissipativeness.
\end{proof}

\begin{theorem}\label{thm_6_4}
	Let $M$ be a complete, regular and visibility manifold with no conjugate points. The geodesic flow $\phi_t$ is conservative with respect to $m$, if and only if the Poincar\'e series diverges at the critical exponent: $\sum_{\alpha\in\Gamma}\mathrm{e}^{-\delta_\Gamma d(p,\alpha p)}=\infty$.
\end{theorem}

\begin{proof}
	If $\phi_t$ is conservative, by Theorem~\ref{thm_6_3}, $\mu_p(L_c(\Gamma))\neq 0$. But by Theorem~\ref{thm_5_5}, the Poincar\'e series converges at the critical exponent implies $\mu_p(L_c(\Gamma))=0$. Thus, $\sum_{\alpha\in\Gamma}\mathrm{e}^{-\delta_\Gamma d(p,\alpha p)}=\infty$.

	For the other direction, assume the Poincar\'e series diverges to infinity, and for any point $p\in\widetilde{M}$. Let $\lambda$ be the maximal possible point mass of $\mu_p$. Follow the claim in Theorem~\ref{thm_4_4}, we can find $R=R(p)>0$, such that for any $x\in\widetilde{M}$ with $d(x,p)>R$, we have
	\begin{displaymath}
		\mu_p(\Sigma_1(x,R))\geq\frac{1}{2}(\mu_p(\widetilde{M}(\infty))-\lambda)\coloneqq C_1>0.
	\end{displaymath}

	Here
	\begin{displaymath}
		\Sigma_1(x,R)=\{\xi\in\widetilde{M}(\infty)\mid\xi=c(+\infty), c(0)\in B(x,R), \exists t>0~\text{s.t.}~c(t)\in B(p,R)\}.
	\end{displaymath}

	Given $\Gamma$ is discrete, there are at most finitely many $\alpha\in\Gamma$ with $\alpha(B(p,R))\cap B(p,R)\neq\emptyset$. Without loss of generality, we can assume that $\alpha(B(p,R))\cap B(p,R)=\emptyset$ for all elements but the identity in $\Gamma$. This is possible because changing finitely many elements does not affect the convergence of the Poincar\'e series. Let $\Gamma B(p,R)$ denote the collection of the image of $B(p,R)$ under $\Gamma$.

	Therefore, for any $\alpha\in\Gamma-\{e\}$, let
	\begin{displaymath}
		\begin{aligned}
			\Sigma_1(\alpha,R) & =\{\xi\in\widetilde{M}(\infty)\mid\xi=c(+\infty), d(c(0),\alpha p)<R, \exists t>0~\text{s.t.}~c(t)\in B(p,R)\}    \\
			\Sigma_2(\alpha,R) & =\{\eta\in\widetilde{M}(\infty)\mid\eta=c(-\infty), d(c(0),\alpha p)<R, \exists t>0~\text{s.t.}~c(t)\in B(p,R)\}.
		\end{aligned}
	\end{displaymath}

	See Figure~\ref{fig6} for a sketch of these sets.

	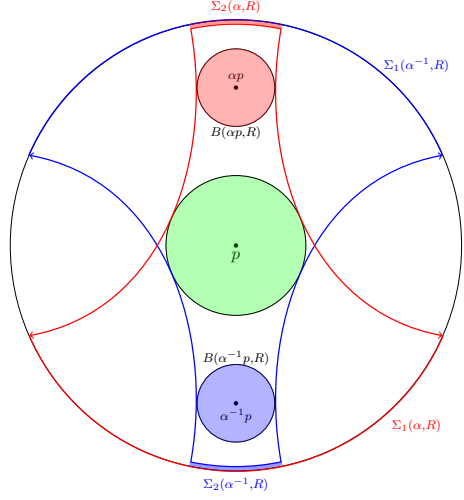
\begin{figure}[htbp]
		\centering
        \scalebox{0.6}{
			\begin{tikzpicture}
				\coordinate [label=below:$p$] (P) at (0,0);
				\node at (P) [circle, fill, inner sep=1pt] {};
				\draw (P) circle (5);
				\draw [fill, green, opacity=0.3] (P) circle (1.55);
				\draw (P) circle (1.55);
				\draw [blue, thick] (4.58,2) arc (23.5:156.5:5);
				\node at (4,4) [blue]{$\scriptstyle\Sigma_1(\alpha^{-1},R)$};
				\draw [red, thick] (-4.58,-2) arc (-156.5:-23.5:5);

				\node at (4,-4) [red]{$\scriptstyle\Sigma_1(\alpha,R)$};
				\draw [red, thick] (1.01,4.8) arc (78:101.8:4.9);
				\draw [red, line width=0.9mm, opacity=0.4] (1.03,4.85) arc (77.5:101.7:4.9);
				\node at (0,5.3) [red]{$\scriptstyle\Sigma_2(\alpha,R)$};
				\draw [blue, thick] (-1.01,-4.8) arc (-101.8:-78:4.9);
				\draw [blue, line width=0.9mm, opacity=0.4] (-1.03,-4.85) arc (-101.7:-77.5:4.9);
				\node at (0,-5.3) [blue]{$\scriptstyle\Sigma_2(\alpha^{-1},R)$};
				\draw [blue, thick, ->] (1,-4.8) to [in=190, out=100] (4.58,2);
				\draw [blue, thick, ->] (-1,-4.8) to [in=-10, out=80] (-4.58,2);
				\draw [red, thick, ->] (1,4.8) to [in=170, out=-100] (4.58,-2);
				\draw [red, thick, ->] (-1,4.8) to [in=10, out=-80] (-4.58,-2);
				\coordinate [label=above:$\scriptstyle\alpha p$] (A) at (0,3.5);
				\node at (A) [circle, fill, inner sep=1pt] {};
				\draw (A) circle (0.86);
				\draw (A) [fill, red, opacity=0.3] circle (0.86);
				\node at (0,2.5) {$\scriptstyle B(\alpha p,R)$};
				\coordinate [label=below:$\scriptstyle\alpha^{-1} p$] (B) at (0,-3.5);
				\node at (B) [circle, fill, inner sep=1pt] {};
				\draw (B) circle (0.86);
				\draw (B) [fill, blue, opacity=0.3] circle (0.86);
				\node at (0,-2.5) {$\scriptstyle B(\alpha^{-1} p,R)$};
			\end{tikzpicture}
		}
		\caption{Sketch of $\Sigma_1$ and $\Sigma_2$.}\label{fig6}
	\end{figure}
	We have $\mu_p(\Sigma_1(\alpha,R))\geq C_1>0$. Observe that $\Sigma_1(\alpha^{-1},R)=\alpha^{-1}\Sigma_2(\alpha,R)$, thus
	\begin{displaymath}
		\begin{aligned}
			\mu_p(\Sigma_2(\alpha,R)) & =\mu_p(\alpha\alpha^{-1}\Sigma_2(\alpha,R))                                                                \\
			                          & =\mu_p(\alpha\Sigma_1(\alpha^{-1},R))                                                                      \\
			                          & =\mu_{\alpha^{-1}p}(\Sigma_1(\alpha^{-1},R))                                                               \\
			                          & =\int_{\Sigma_1(\alpha^{-1},R)}\mathrm{e}^{-\delta_\Gamma\beta_\xi(\alpha^{-1}p,p)}\,\mathrm{d}\mu_p(\xi).
		\end{aligned}
	\end{displaymath}

	For any $\xi\in\Sigma_1(\alpha^{-1},R)$, by definition, we can find a point $q\in\widetilde{M}$ and $t_1>0$, such that $d(q,\alpha^{-1}p)<R$ and $c_{q,\xi}(t_1)\in B(p,R)$. By Proposition~\ref{prop_2_7}~(1), we have 
    \begin{displaymath}
        \begin{aligned}
            d(c_{\alpha^{-1}p,\xi}(t_1),c_{q,\xi}(t_1))
            &\leq 2R(\frac{\pi}{2})+3d(\alpha^{-1}p,q)\\
            &\leq 2R(\frac{\pi}{2})+3R.
        \end{aligned}
    \end{displaymath}
    
    By triangle inequality, we have 
    \begin{displaymath}
       d(c_{\alpha^{-1}p,\xi},p)\leq             d(c_{\alpha^{-1}p,\xi}(t_1),c_{q,\xi}(t_1))+d(c_{q,\xi}(t_1),p)\leq 2R(\frac{\pi}{2})+4R.
    \end{displaymath}

 Thus, $\xi\in\text{pr}_{\alpha^{-1}p}B(p,2R(\frac{\pi}{2})+4R)$.

	By Lemma~\ref{lem_4_2}, we have
	\begin{displaymath}
        d(p,\alpha p)-4(R(\frac{\pi}{2})+2R)\leq\beta_\xi(\alpha^{-1}p,p)\leq d(p,\alpha p).
	\end{displaymath}

	Replace $\beta_\xi$ part in the previous equation, we have the measure of $\Sigma_2(\alpha,R)$ is bounded by
	\begin{equation}\label{equ_6_3}
		C_1\mathrm{e}^{-\delta_\Gamma d(p,\alpha p)}\leq\mu_p(\Sigma_2(\alpha,R))\leq C_2\mathrm{e}^{-\delta_\Gamma d(p,\alpha p)}.
	\end{equation}

    Here $C_2=\mathrm{e}^{4(R(\frac{\pi}{2})+2R)\delta_\Gamma}\mu_p(\widetilde{M}(\infty))$ is a constant, independent of the choice of points.

    Now given $(\xi,\eta)\in\Sigma_1(\alpha,R)\times\Sigma_2(\alpha,R)$ and let $\widetilde{T}_R=T^1B(p,R)\subset T^1\widetilde{M}$. We want to estimate a lower bound for $\text{Length}(c_{\xi,\eta}\cap\pi(\widetilde{T}_R\cap\phi_{-d(p,\alpha p)}(\alpha\widetilde{T}_R)))$.

	Not hard to observe that
	\begin{displaymath}
		B(p,2R)\subset\pi(\widetilde{T}_{16R}\cap\phi_{-d(p,\alpha p)}(\alpha\widetilde{T}_{16R})),
	\end{displaymath}
	and for $k_2>k_1>1$, for any $\alpha\in\Gamma-\{e\}$,
	\begin{equation}\label{equ_6_4}
		\Sigma_1(\alpha,k_1R)\times\Sigma_2(\alpha,k_1R)\subset\Sigma_1(\alpha,k_2R)\times\Sigma_2(\alpha,k_2R).
	\end{equation}

	Thus, for any
	\begin{displaymath}
		(\xi,\eta)\in\Sigma_1(\alpha,R)\times\Sigma_2(\alpha,R)\subset\Sigma_1(\alpha,16R)\times\Sigma_2(\alpha,16R).
	\end{displaymath}

	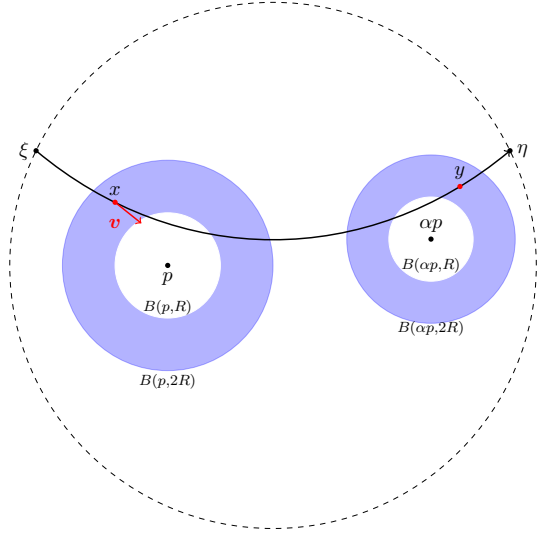
\begin{figure}[htbp]
		\centering
		\scalebox{0.7}{
			\begin{tikzpicture}
				\draw [dashed] (0,0) circle (5);
				\coordinate [label=left:$\xi$] (X) at (-4.5,2.18);
				\node at (X) [circle, fill, inner sep=1pt] {};
				\coordinate [label=right:$\eta$] (Y) at (4.5,2.18);
				\node at (Y) [circle, fill, inner sep=1pt] {};
				\draw (-2,0) [fill, blue, opacity=0.3] circle (2);
				\draw (-2,0) [fill, white] circle (1);
				\draw (3,0.5) [fill, blue, opacity=0.3] circle (1.6);
				\draw (3,0.5) [fill, white] circle (0.8);
				\coordinate [label=below:$p$] (A) at (-2,0);
				\node at (A) [circle, fill, inner sep=1pt] {};
				\coordinate [label=above:$\alpha p$] (B) at (3,0.5);
				\node at (B) [circle, fill, inner sep=1pt] {};

				\draw [->, thick] (X) to [in=220, out=-40](Y);
				\node at (-2,-2.2) {$\scriptstyle B(p,2R)$};
				\node at (-2,-0.8) {$\scriptstyle B(p,R)$};
				\node at (3,-1.2) {$\scriptstyle B(\alpha p,2R)$};
				\node at (3,0) {$\scriptstyle B(\alpha p,R)$};
				\coordinate [label=above:$x$] (X) at (-3,1.2);
				\node at (X) [red, circle, fill, inner sep=1pt] {};
				\draw [red, ->, thick] (X) to (-2.5,0.8);
				\node at (-3,0.8) [red] {$\bm{v}$};
				\coordinate [label=above:$y$] (Y) at (3.55,1.5);
				\node at (Y) [red, circle, fill, inner sep=1pt] {};
			\end{tikzpicture}
		}
		\caption{Choice of $x$ and $y$.}\label{fig7}
	\end{figure}

	As shown in Figure~\ref{fig7}, pick a connecting geodesic $c_{\xi,\eta}$, and $x\in c_{\xi,\eta}\cap B(p,2R)$. Parametrize the geodesic such that $c_{\xi,\eta}(0)=x$ and $\bm{v}=c'_{\xi,\eta}(0)$. Choose another point on the geodesic $y\in c_{\xi,\eta}\cap B(\alpha p,2R)$, followed by the triangle inequality,
	\begin{displaymath}
		d(p,\alpha p)-4R\leq d(x,y)\leq d(p,\alpha p)+4R.
	\end{displaymath}

	We could conclude that $d(c_{\bm{v}}(d(p,\alpha p)),\alpha p)\leq 8R$, thus
	\begin{displaymath}
		\bm{v}\in\widetilde{T}_{8R}\cap\phi_{-d(p,\alpha p)}(\widetilde{T}_{8R})\subset \widetilde{T}_{16R}\cap\phi_{-d(p,\alpha p)}(\widetilde{T}_{16R}).
	\end{displaymath}

	Project the geodesic segment to $T^1M$, the length has a lower bound
	\begin{displaymath}
		\text{Length}(\pi(c'_{\xi,\eta}\cap\widetilde{T}_{16R}\cap\phi_{-d(p,\alpha p)}(\widetilde{T}_{16R})))\geq 2R.
	\end{displaymath}

    We can find $A=A(R,\delta_0)$, such that
	\begin{equation}\label{equ_6_5}
		\begin{aligned}
			 & \widetilde{m}(\widetilde{T}_{16R}\cap\phi_{-d(p,\alpha p)}(\alpha\widetilde{T}_{16R}))                                                                                                                 \\
             & \geq\int_{\Sigma_1(\alpha,R)\times\Sigma_2(\alpha,R)}\text{Length}(c_{\xi,\eta}\cap\pi(\widetilde{T}_{16R}\cap\phi_{-d(p,\alpha p)}(\alpha\widetilde{T}_{16R})))\,\mathrm{d}\widetilde{\mu}(\xi,\eta) \\
			 & \geq A\mathrm{e}^{-\delta_\Gamma d(p,\alpha p)}.
		\end{aligned}
	\end{equation}

	Choose any $R>1$, for $\alpha\in\Gamma$, let $n\in\mathbb{N}$ such that $n\leq d(p,\alpha p)<n+1$. Similar to Equation~\ref{equ_6_5}, we can find $D=D(R,\delta_0)>0$, such that
	\begin{equation}\label{equ_6_6}
		\widetilde{m}(\widetilde{T}_{16R}\cap\phi_{-n}(\alpha\widetilde{T}_{16R}))\geq D\mathrm{e}^{-\delta_\Gamma d(p,\alpha p)}.
	\end{equation}

	Because the Poincar\'e series diverges, we can conclude that by taking $\alpha$ throughout $\Gamma$,
	\begin{displaymath}
		\sum_{n=1}^\infty    \widetilde{m}(\widetilde{T}_{16R}\cap\phi_{-n}(\Gamma\widetilde{T}_{16R}))\geq D\sum_{\alpha\in\Gamma}\mathrm{e}^{-\delta_\Gamma d(p,\alpha p)}=+\infty.
	\end{displaymath}

	Let $\widetilde{E}_n=\widetilde{T}_{16R}\cap\phi_{-n}(\Gamma\widetilde{T}_{16R})$ and define the probability of this set by the ratio
	\begin{displaymath}
		P(\widetilde{E}_n)=\frac{\widetilde{m}(\widetilde{T}_{16R}\cap\phi_{-n}(\Gamma\widetilde{T}_{16R}))}{\widetilde{m}(\widetilde{T}_{16R})}.
	\end{displaymath}

	We can check that $P(\widetilde{E}_{n+k}|\widetilde{E}_n)\leq P(\widetilde{E}_k)$. In fact, for any isometry $\alpha\in\Gamma$ such that $\alpha(\widetilde{T}_{16R})\cup\phi_{n}(\widetilde{T}_{16R})\neq\emptyset$, we have
	\begin{displaymath}
		\begin{aligned}
			\frac{\widetilde{m}(\alpha(\widetilde{T}_{16R})\cap\phi_n(\widetilde{T}_{16R})\cap\phi_{-k}(\Gamma\widetilde{T}_{16R}))}{\widetilde{m}(\alpha\widetilde{T}_{16R})} & \leq\frac{\widetilde{m}(\alpha(\widetilde{T}_{16R})\cap\phi_{-k}(\Gamma\widetilde{T}_{16R}))}{\widetilde{m}(\widetilde{T}_{16R})} \\
			                                                                                                                                                                   & =\frac{\widetilde{m}(\widetilde{T}_{16R}\cap\phi_{-k}(\Gamma\widetilde{T}_{16R}))}{\widetilde{m}(\widetilde{T}_{16R})}            \\
			                                                                                                                                                                   & =P(\widetilde{E}_k)
		\end{aligned}
	\end{displaymath}

	It follows that $P(\widetilde{E}_{n+k}\cap \widetilde{E}_n)=P(\widetilde{E}_{n+k}|\widetilde{E}_n)P(\widetilde{E}_n)\leq P(\widetilde{E}_n)P(\widetilde{E}_k)$. By Borel-Cantelli Lemma, $ P(\widetilde{E}_k\text{ i.\ o.})=1$. We can find $n_k\to\infty$ for $\widetilde{m}$-a.e. $\bm{v}\in\widetilde{T}_{16R}\subset T^1\widetilde{M}$, and $\alpha_{n_k}$ depends on $\bm{v}$, such that
	\begin{displaymath}
		\phi_{n_k}(\bm{v})\in\alpha_{n_k}(\widetilde{T}_{16R}).
	\end{displaymath}

	This suggests the geodesic $c_{\bm{v}}$ in $\widetilde{M}$ will intersect with $\Gamma(B(p,16R))$ for infinitely many times. Thus, on $T^1M$, the projection of the geodesic $\pi c_{\bm{v}}$ will intersect with the compact ball $\pi(\Gamma(B(p,16R)))$ for infinitely many times, namely $\pi\bm{v}$ is a conservative point.

	Therefore, the geodesic flow $\phi_t$ is conservative with respect to $m$.
\end{proof}

Next, we will prove that conservativity of the geodesic flow with respect to $m$ implies ergodicity. The following proposition is needed in the proof.

\begin{proposition}\label{prop_6_5}
    Let $M$ be a complete, regular and uniform visibility manifold with no conjugate points. Let $\xi\in L_c(\Gamma)$ and $p\in\widetilde{M}$, $\bm{v}\in T^1\widetilde{M}$ with $c_{\bm{v}}(+\infty)=\xi$.

    Then, for any $\bm{w}$ positive asymptote to $\bm{v}$, we can find some $a\in\mathbb{R}$, such that 
    \begin{displaymath}
        d(c_{\bm{v}}(t+a),c_{\bm{w}}(t))\to0,\qquad t\to +\infty
    \end{displaymath}

    Similar results hold for vectors negatively asymptotic.
\end{proposition}

\begin{proof}
    Given $\xi$ is a conical limit point, we can find some $D>0$ and $\alpha_n\in\Gamma$, such that
    \begin{displaymath}
        d(\alpha_n p, c_{\bm{v}}\vert_{(0,+\infty)})\leq D, \quad n\in\mathbb{N}.
    \end{displaymath}

    We can choose $t_n\to+\infty$, such that $d(\alpha_n p,c_{\bm{v}}(t_n))\leq D$. Using $\alpha_n^{-1}$ to pull the point $c_{\bm{v}}(t_n)$ back to the $D$ neighborhood of the point $p$, together with $\mathrm{d}\alpha_n^{-1}(c_{\bm{v}}'(t_n))$. 
    
    We know that $T^1B(p,D)$ is compact, by passing to a subsequence if needed, assume that $\mathrm{d}\alpha_n^{-1}(c'_{\bm{v}}(t_n))$ converges to some $\bar{\bm{v}}\in T^1\widetilde{M}$.

    Now given $\bm{w}$ positive asymptote to $\bm{v}$, by Proposition~\ref{prop_2_7}~(1), we have
    \begin{displaymath}
        d(c_{\bm{v}}(t),c_{\bm{w}}(t))\leq 2R(\frac{\pi}{2})+3d(c_{\bm{v}}(0),c_{\bm{w}}(0))\triangleq B.
    \end{displaymath}

    By triangle inequality, we have
    \begin{displaymath}
        d(\alpha_n^{-1}c_{\bm{w}}(t_n),p)\leq B+D.
    \end{displaymath}

    Again, by passing to a subsequence if needed, we can assume that $\mathrm{d}\alpha_n^{-1}(c_{\bm{w}}'(t_n))$ converges to $\bar{\bm{w}}$.

    Thus, for any $s>-t_n$, we have
    \begin{displaymath}
        d(c_{\bm{v}}(s+t_n),c_{\bm{w}}(s+t_n))=d(c_{\mathrm{d}\alpha_n^{-1}c_{\bm{v}(s+t_n)}'}(s+t_n),c_{\mathrm{d}\alpha_n^{-1}c_{\bm{w}(s+t_n)}'}(s+t_n)).
    \end{displaymath}

    Let $n\to +\infty$, we can conclude $d(c_{\bar{\bm{v}}}(s),c_{\bar{\bm{w}}}(s))\leq B$ for any $s$. So $\bar{\bm{v}}$ and $\bar{\bm{w}}$ are bi-asymptotic. Because $M$ is regular, the connecting geodesic should be unique. We can find some $a$, such that $c_{\bar{\bm{v}}}(t+a)= c_{\bar{\bm{w}}}(t)$. This $a$ is what we need. As we can see, given any $\epsilon>0$, we have for $n$ sufficiently large, by the continuity of the geodesics
    \begin{displaymath}
        \begin{aligned}
            &d(c_{\bm{v}}(t+t_n+a),c_{\bm{w}}(t+t_n))=d(\alpha_n^{-1}c_{\bm{v}}(t+t_n+a),\alpha_n^{-1}c_{\bm{w}}(t+t_n))\\
                                                    &\leq d(\alpha_n^{-1}c_{\bm{v}}(t+t_n+a),c_{\bar{\bm{v}}}(t+a))+d(c_{\bar{\bm{v}}}(t+a),c_{\bar{\bm{w}}}(t))+d(c_{\bar{\bm{w}}}(t),\alpha_n^{-1}c_{\bm{w}}(t+t_n))\\
                                                    &\leq \epsilon+0+\epsilon\\
                                                    &=2\epsilon.   
        \end{aligned}
    \end{displaymath}
\end{proof}

\begin{theorem}\label{thm_6_6}
	Let $M$ be a complete, regular and uniform visibility manifold with no conjugate points and $m$ is the BMS measure. If the geodesic flow $\phi_t$ is conservative with respect to $m$, then it is ergodic with respect to $m$.
\end{theorem}

\begin{proof}
   Choose a positive function $\rho\in L^1(m)$ such that for any $\bm{v}$ and $\bm{w}$ in $T^1 M$ with $d(c_{\bm{v}}(0),c_{\bm{w}}(0))<1$, 
   \begin{displaymath}
       \frac{\rho(\bm{v})-\rho(\bm{w})}{\rho(\bm{w})}\leq C_1 d(c_{\bm{v}}(0),c_{\bm{w}}(0)).
   \end{displaymath}
   Here $C_1>0$ is a constant. The standard construction of such function can be found in~\cite{Ni}.

It is well known that for any $\bm{v}\in M_C$ and $\rho>0$, we have 
\begin{displaymath}
    \int_0^{+\infty}\rho(c_{\bm{v}}(t))\,\mathrm{d}t=+\infty.
\end{displaymath}

Now, for any $f\in L^1(m)$, by the Hopf ergodic theorem,

\begin{displaymath}
f_{\rho}^{+}(\bm{v})\triangleq \lim_{T\to +\infty} \frac{\int_0^T f(c_{\bm{v}}(t))\,\mathrm{d}t}{\int_0^T \rho(c_{\bm{v}}(t))\,\mathrm{d}t}.
\end{displaymath}

The limit $f_{\rho}^{+}(\bm{v})$ exists and invariant under the geodesic flows, for $m$-a.e.\ vector $\bm{v}$.

Next, we want to show that for $m$-a.e.\ $\bm{w}$ positive asymptote to $\bm{v}$, we have $f_{\rho}^{+}(\bm{v})=f_{\rho}^{+}(\bm{w})$. By the above proposition, we can find $a$ such that $d(c_{\bm{v}}(t+a),c_{\bm{w}}(t))\to 0$ as $t\to +\infty$. A direct computation shows that
\begin{displaymath}
\begin{aligned}
   f_{\rho}^{+}(\bm{v})-f_{\rho}^{+}(\bm{w})
   =&\lim_{T\to +\infty}\left(\frac{\int_0^T f(c_{\bm{v}}'(t+a))\,\mathrm{d}t}{\int_0^T \rho(c_{\bm{v}}'(t+a))\,\mathrm{d}t}-\frac{\int_0^T f(c_{\bm{w}}'(t))\,\mathrm{d}t}{\int_0^T \rho(c_{\bm{w}}'(t))\,\mathrm{d}t}\right)\\
   =&\lim_{T\to\infty}\left\{\frac{\int_0^T[f(c_{\bm{v}}'(t+a))-f(c_{\bm{w}}'(t))]\,\mathrm{d}t}{\int_0^T \rho(c_{\bm{v}}'(t+a))\,\mathrm{d}t}+\right.\\
    &\left. \frac{\int_0^T f(c_{\bm{v}}'(t))\,\mathrm{d}t}{\int_0^T \rho(c_{\bm{w}}'(t))\,\mathrm{d}t}\frac{\int_0^T[\rho(c_{\bm{w}}'(t))-\rho(c_{\bm{v}}'(t+a))]\,\mathrm{d}t}{\int_0^T \rho(c_{\bm{v}}'(t+a))\,\mathrm{d}t}\right\}
\end{aligned}        
\end{displaymath}

By the choice of $\rho$, we can see this limit is zero for $f$ continuous with compact support. Thus, we can argue for $f\in L^1(m)$, $f_{\rho}^{+}(\bm{v})=f_{\rho}^{+}(\bm{w})$ as required. 

Lastly, lift $f_{\rho}^{+}$ to $T^1\widetilde{M}$, denoted by $\tilde{f}_{\rho}^{+}$. Under the standard Hopf coordinates~(cf.~\cite{Ni}), the BMS measure $\tilde{m}$ is a product measure defined on $\widetilde{M}^2(\infty)\times\mathbb{R}$. By the construction of $\tilde{f}_{\rho}^{+}$, it is $\Gamma$-invariant on the first two coordinates, and $\tilde{\phi}_t$ invariant on the third coordinate. Thus, for any $\xi\in L_C(\Gamma)$, $\tilde{f}_{\rho}^{+}$ is constant on $\{\xi\}\times\widetilde{M}(\infty)$ and $\widetilde{M}(\infty)\times\{\xi\}$. Together with $\tilde{f}_{\rho}^{+}$ is $\Gamma$ invariant, we can show that $\tilde{f}_{\rho}^{+}$ is $\mu_p\times\mu_p$-a.e.\ constant, thus is $\tilde{m}$-a.e.\ constant. The Hopf ergodic theorem implies that the geodesic flow is ergodic.
\end{proof}

For the other direction, we have the following result:

\begin{theorem}\label{thm_6_7}
    Let $M$ be a complete, regular and uniform visibility manifold with no conjugate points and $\Gamma$ be a non-elementary discrete subgroup of $\text{Iso}(\widetilde{M})$. Suppose $\mu={\{\mu_q\}}_{q\in\widetilde{M}}$ is the Patterson-Sullivan measure, and $p\in\widetilde{M}$ is a point chosen.

	Then $\Gamma$ acts on $\widetilde{M}(\infty)\times\widetilde{M}(\infty)$ is ergodic with respect to $\mu_p\times\mu_p$, if and only if $\phi_t$ is ergodic with respect to the BMS measure $m$.
\end{theorem}

\begin{proof}
	Suppose $\Gamma$ is not ergodic, we can find a set $\widetilde{E}\subset\widetilde{M}(\infty)\times\widetilde{M}(\infty)$ with positive measure, and its complement also have positive measure, with respect to $\mu_p\times\mu_p$.

	Define $\widetilde{E}'\subset T^1\widetilde{M}$ as following:
	\begin{displaymath}
        \widetilde{E}'=\{\bm{v}\mid (\xi,\eta)\in \widetilde{E}, c_{\bm{v}}(+\infty)=\xi, c_{\bm{v}}(-\infty)=\eta\}.
	\end{displaymath}

	For each pair of $(\xi,\eta)$, all the unit tangent vectors along the whole connecting geodesic are contained in $\widetilde{E'}$.

	By the construction of BMS measure, $(\mu_p\times\mu_p)(\widetilde{E})>0$ implies $\widetilde{m}(\widetilde{E'})>0$. Project it onto the $T^1M$, say $E'$, with positive measure $m(E')>0$. Clearly $E'$ is $\phi_t$ invariant. A similar argument applies to the complement of $\widetilde{E}$. It follows that $\phi_t$ is not ergodic with respect to $m$.

	The other direction is straightforward by reversing the argument.
\end{proof}

As a summary of everything above, we state our main theorem here:

\begin{theorem}\label{thm_6_8}
    Let $M$ be a complete, regular and uniform visibility manifold with no conjugate points, $\mu={\{\mu_q\}}_{q\in\widetilde{M}}$ is the Patterson-Sullivan measure, $p\in\widetilde{M}$ is arbitrarily chosen and $m$ is the BMS measure. Then the followings are equivalent:
	\begin{enumerate}
		\item[1.] $\mu_p(L_c(\Gamma))=\mu_p(\widetilde{M}(\infty))$.
		\item[2.] $\phi_t$ is conservative with respect to $m$.
		\item[3.] $\phi_t$ is ergodic with respect to $m$.
		\item[4.] The $\Gamma$-action on $\widetilde{M}(\infty)\times\widetilde{M}(\infty)$ is ergodic with respect to $\mu_p\times\mu_p$.
		\item[5.] The Poincar\'e series $\sum_{\alpha\in\Gamma}\mathrm{e}^{-\delta_\Gamma d(p,\alpha p)}$ diverges.
	\end{enumerate}
\end{theorem}

\begin{proof}
	$1\Leftrightarrow 2$: Theorem~\ref{thm_6_3}.

	$2\Rightarrow 3$: Theorem~\ref{thm_6_6}.

	$3\Rightarrow 2$: Straightforward.

	$3\Leftrightarrow 4$: Theorem~\ref{thm_6_7}.

	$2\Leftrightarrow 5$: Theorem~\ref{thm_6_4}.
\end{proof}

Similarly, based on the HTS dichotomy given in Theorem~\ref{thm_6_3}, the following result follows directly from the above theorem.

\begin{theorem}\label{thm_6_9}
	Under the same setting of Theorem~\ref{thm_6_8}, the followings are equivalent.
	\begin{enumerate}
		\item[1.] $\mu_p(L_c(\Gamma))=0$.
		\item[2.] $\phi_t$ is completely dissipative with respect to $m$.
		\item[3.] $\phi_t$ is not ergodic with respect to $m$.
		\item[4.] The $\Gamma$-action on $\widetilde{M}(\infty)\times\widetilde{M}(\infty)$ is not ergodic with respect to $\mu_p\times\mu_p$.
		\item[5.] The Poincar\'e series $\sum_{\alpha\in\Gamma}\mathrm{e}^{-\delta_\Gamma d(p,\alpha p)}$ converges.
	\end{enumerate}
\end{theorem}

\section{Results on Manifold with No Focal Points}\label{sec7}

In this section, we will show that our results hold on manifolds with no focal points, in considerably different context where flat strips may exist. 

We call two points $p_k=c(t_k), k=0,1$ on a geodesic $c$ are \emph{focal}, if there is a non-trivial Jacobi field $\bm{J}$, such that
\begin{displaymath}
    \bm{J}(t_0)=0,\quad \bm{J}'(t_0)\neq 0,\quad \frac{\mathrm{d}}{\mathrm{d} {t}}\Big|_{t=t_1} {\Vert \bm{J}(t)\Vert}^2=0.
\end{displaymath}

From the definition we can see, the length of the Jacobi field is strictly increasing along $c$ for $t>t_0$ if there is no focal points. Easy to see, manifolds with non-positive sectional curvature has no focal points. But the converse is not true, there are examples of manifolds with no focal points but admit positive sectional curvature. Therefore, it is a non-trivial generalization of manifolds of non-positive sectional curvature. 

Another observation is that conjugates points implies focal points. This suggests that manifolds with no focal points is of a narrower scope than manifolds with no conjugate points. In fact, manifolds with no focal points provide more geometric tools. For example, Proposition~\ref{prop_2_6} holds automatically without the requirement of visibility (cf.~\cite{OS1}). 

It is natural to expect our results hold on manifolds with no focal points under weaker conditions. We managed to remove the need that the manifolds should be regular. That is to say, for any two different points $\xi\neq\eta$ on the ideal boundary, there may exist more than one (at least one) connecting geodesic. 

This is not a trivial generalization because the Flat Strip Theorem (cf.~\cite{OS1}) tells us two bi-asymptotic geodesics bound a flat strip, which is an obstacle to the hyperbolic behavior of the geodesic flow. Thus, we need to put some restriction on the flat strip: we require that the width (distance between bounding geodesics) of the flat strip has a positive lower bound $\delta>0$. We should claim here that this restriction is reasonable. Examples are provided in~\cite{LWW2}, such as surfaces with no focal points whose subset of points with negative curvature has only finitely many connected components. This kind of surfaces is known as the cases with the mildest requirement, so far, provides a positive result to the famous conjecture on the ergodicity of the geodesic flow with respect to the Liouville measure. 

Throughout this section, we assume $M$ is a complete visibility manifolds with no focal points, and all flat strips (if any) has a common positive lower bound $\delta>0$. As a geometric preparation, we state the following properties hold under the new setting in analogy to Proposition~\ref{prop_2_7}:

\begin{proposition}\label{prop_7_1}
	Assume $\widetilde{M}$ to be the universal cover of a manifold $M$ with no focal points.
	\begin{enumerate}
		\item\cite{OS1} Assume there are two positively asymptotic geodesics $c_1$ and $c_2$, then the distance function
		\begin{displaymath}
			\begin{aligned}
				f:~& \mathbb{R}^{+} & \to     &~\mathbb{R}^+,         \\
				    & t              & \mapsto &~f(t)=d(c_1(t),c_2(t))
			\end{aligned}
		\end{displaymath}
		is a non-increasing function.
		\item\cite{LWW} The following map is continuous:
		\begin{displaymath}
			\begin{aligned}
				\Psi:~& T^1\widetilde{M}\times[-\infty,\infty] & \to     &~\widetilde{M}\cup\widetilde{M}(\infty), \\
				       & \qquad(\bm{v},t)                       & \mapsto &~c_{\bm{v}}(t).
			\end{aligned}
		\end{displaymath}
		\item\cite{LWW} For any $\bm{v}\in T^1\widetilde{M}$ and $R,\epsilon$ positive, there is a constant $L=L(\bm{v},\epsilon,R)$ such that for any $t>L$, we have
		\begin{displaymath}
			B(c_{\bm{v}}(t),R)\subset C(\bm{v},\epsilon).
		\end{displaymath}
		Here $B(c_{\bm{v}}(t),R)$ is the open ball centered at $c_{\bm{v}}(t)$ with radius $R$.
	\end{enumerate}
\end{proposition}

As we know, manifolds with no focal points must contain no conjugate points, all the conclusions in our article still hold if the regular condition is not needed. Thus, the shadow lemmas in Section~\ref{sec4} and the properties of conical limit sets in Section~\ref{sec5} are still valid. 

The first thing needs modification to adapt to the possible existence of the flat strips is the construction of the BMS measure. On the unit tangent bundle of the universal cover $\widetilde{M}$, we define   

\begin{displaymath}
	\widetilde{m}(A)=\int_{\widetilde{M}^2(\infty)}\text{Vol}(\pi(P^{-1}(\xi,\eta)\cap A))\,\mathrm{d}\widetilde{\mu}(\xi,\eta)
\end{displaymath}
for any Borel set $A\subset T^1\widetilde{M}$.

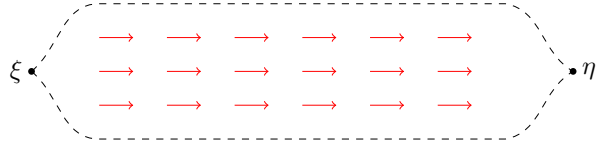
\begin{figure}[htbp]
	\centering
	\scalebox{0.9}{
		\begin{tikzpicture}
			\coordinate [label=left:$\xi$] (X) at (-4,0);
			\node at (X) [circle, fill, inner sep=1pt]{};
			\coordinate [label=right:$\eta$] (Y) at (4,0);
			\node at (Y) [circle, fill, inner sep=1pt]{};
			\draw [dashed] (X) to [in=180, out=40] (-3,1) to (3,1) [in=180, out=0] to [in=150, out=-10](Y);
			\draw [dashed] (X) to [in=180, out=-40] (-3,-1) to (3,-1) [in=180, out=0] to [in=-150, out=10](Y);
			\draw [red, ->] (-3,0.5) to (-2.5,0.5);
			\draw [red, ->] (-2,0.5) to (-1.5,0.5);
			\draw [red, ->] (-1,0.5) to (-0.5,0.5);
			\draw [red, ->] (0,0.5) to (0.5,0.5);
			\draw [red, ->] (1,0.5) to (1.5,0.5);
			\draw [red, ->] (2,0.5) to (2.5,0.5);
			\draw [red, ->] (-3,0) to (-2.5,0);
			\draw [red, ->] (-2,0) to (-1.5,0);
			\draw [red, ->] (-1,0) to (-0.5,0);
			\draw [red, ->] (0,0) to (0.5,0);
			\draw [red, ->] (1,0) to (1.5,0);
			\draw [red, ->] (2,0) to (2.5,0);
			\draw [red, ->] (-3,-0.5) to (-2.5,-0.5);
			\draw [red, ->] (-2,-0.5) to (-1.5,-0.5);
			\draw [red, ->] (-1,-0.5) to (-0.5,-0.5);
			\draw [red, ->] (0,-0.5) to (0.5,-0.5);
			\draw [red, ->] (1,-0.5) to (1.5,-0.5);
			\draw [red, ->] (2,-0.5) to (2.5,-0.5);
		\end{tikzpicture}
	}
	\caption{$P^{-1}(\xi,\eta)$ on the Flat Strip.}\label{fig8}
\end{figure}

Instead of using the length of the connecting geodesic segment, we use $\text{Vol}$ here, which stands for the volume on the manifold $\widetilde{M}$. The map $P:T^1\widetilde{M}\to\widetilde{M}^2(\infty)$ sends a vector $\bm{\widetilde{v}}$ to the pair $(c_{\bm{\widetilde{v}}}(+\infty),c_{\bm{\widetilde{v}}}(-\infty))$. Given any pair $(\xi,\eta)$, if there is a unique connecting geodesic of these two points, we can define $P^{-1}(\xi,\eta)$ be the set of all unit tangent vectors along the connecting geodesic from $\xi$ to $\eta$. If the connecting geodesics are not unique, these geodesics will form a flat strip, and we define $P^{-1}(\xi,\eta)$ be the collection of all unit tangent vectors pointing to $\eta$ on the flat strip. See Figure~\ref{fig8}. Easy to check this is well-defined and invariant under the geodesics and $\Gamma$-action.  

Here, the minimal width condition guarantees that the flat strip with positive length has positive volume. Thus, in the proof of Theorem~\ref{thm_6_4}, we can still find some constant $A$ for which the Equation~(\ref{equ_6_5}) holds.

Another results involving the BMS measure is the proof of the conservativity of the geodesic flow $\phi_t$ with respect to $m$ implies that $\phi_t$ is ergodicity under our setting. In the original proof in Section~\ref{sec6}, we actually use the Hopf argument, which may fails when there are flat strips. Thus, we give another proof here.   

We start with the definition of non-wandering set on $T^1M$. Given a vector $\bm{v}\in T^1M$, let:
\begin{displaymath}
	P^{+}(\bm{v})=\{\bm{w}\in T^1M\mid \exists \{\bm{v}_n\}\to\bm{v},~t_n\to +\infty,~\text{s.t.}~\phi_{t_n}(\bm{v}_n)\to\bm{w}\}.
\end{displaymath}

We call $\bm{v}$ is a \emph{non-wandering point}, if $\bm{v}\in P^{+}(\bm{v})$, and the collection of all non-wandering points is called the \emph{non-wandering set}:
\begin{displaymath}
	\Omega=\{\bm{v}\in T^1M\mid \bm{v}\in P^{+}(\bm{v})\}.
\end{displaymath}

The following theorem is a well-known result characterize the non-wandering set on manifolds with non-positive curvature by Eberlein:

\begin{theorem}[\cite{Eb2}]\label{thm_7_2}
    Let $M$ be a visibility manifold with non-positive curvature, $\Gamma$ is a non-elementary discrete subgroup of $\text{Iso}(\widetilde{M})$. Let $\Omega$ be the non-wandering set of the geodesic flow, then $\bm{v}\in\Omega$ if and only if $c_{\bm{\widetilde{v}}}(\pm\infty)\in L(\Gamma)$. Here $\bm{\widetilde{v}}$ is any lift of $\bm{v}$.
\end{theorem}

This characterization plays an important role in the discussion on HTS dichotomy, but the proof given by Eberlein relied on properties of manifolds with non-positive curvature that no longer holds under our setting. To fill the gap, an extra assumption, the uniform visibility of the manifolds, is needed. We highlight the part require this extra assumption as in the following lemma.

\begin{lemma}\label{lem_7_3}
	Let $M$ be a complete, uniform visibility manifold with no focal points and $\Gamma$ is a non-elementary discrete subgroup of $\text{Iso}(\widetilde{M})$. Given $\xi\in\widetilde{M}(\infty)$ with a converging sequence of point $p_n\to\xi$ in $\widetilde{M}$ under the cone topology, we have for any open neighborhood $U$ of $\xi$ in $\overline{\widetilde{M}}=\widetilde{M}\cup\widetilde{M}(\infty)$,
	\begin{displaymath}
		\measuredangle_{p_n}(\overline{\widetilde{M}}-U)\to 0,\quad n\to\infty.
	\end{displaymath}
\end{lemma}

\begin{proof}
	Under the cone topology, we can choose a point $p\in\widetilde{M}$ not in the sequence $p_n$ and let $\bm{v}=c'_{p,\xi}(0)$ be the unit tangent vector of the connecting geodesic from $p$ to $\xi$, such that $U=C(\bm{v},\epsilon)$ for some $\epsilon>0$.

	It is sufficient to show that for arbitrary sequence of point $\{x_n\}\not\subset U$, we have $\theta_n=\measuredangle_{p_n}(p,x_n)\to 0$. Let $c_{x_n,p_n}$ denote the connecting geodesic segment from $x_n$ to $p_n$. See Figure~\ref{fig9} for a sketch.

	\begin{figure}[htbp]
		\centering
		\begin{tikzpicture}
			\coordinate[label=right:$p$](P) at (0,0);
			\node[circle, fill, inner sep=1pt] at (P) {};
			\draw (P) circle (4);
			\draw [thick, fill=cyan!20] (P) to (-3.9,0.89) arc (167.2:192.8:4) to (P);
			\node at (-2,0.8) {$\scriptstyle C(\bm{v},\epsilon)$};
			\coordinate[label=left:$\xi$](X) at (-4,0);
			\node[circle, fill, inner sep=1pt] at (X) {};
			\draw [red, ->] (P) to (X);
			\draw [red, thick, ->] (P) to (-0.5,0);
			\node [red] at (-0.5,0.3) {$\bm{v}$};
			\coordinate[label=right:$x_n$](A) at (2,-2);
			\node[circle, fill, inner sep=1pt] at (A) {};
			\coordinate[label=left:$\scriptstyle p_n$] (B) at (-3.2,-0.4);
			\node[circle, fill, inner sep=1pt] at (B) {};
			\draw [blue] (A) to [in=-10, out=130](B);

			\coordinate[label=below:$q_n$] (Q) at (-0.1,-0.78);
			\node[circle, fill, inner sep=1pt] at (Q) {};
			\draw [->] (P) to (Q);
			\draw [blue] (0.08,-0.82) to (0.11,-0.68) to (-0.08,-0.64);
			\draw [blue] (P) to (B);
			\draw [red] (-2.7,-0.49) arc (-10:8:0.45);
			\node [red] at (-2.5,-0.75) {$\scriptstyle \theta_n$};
		\end{tikzpicture}
		\caption{Measure the Angle at $p_n$.}\label{fig9}
	\end{figure}
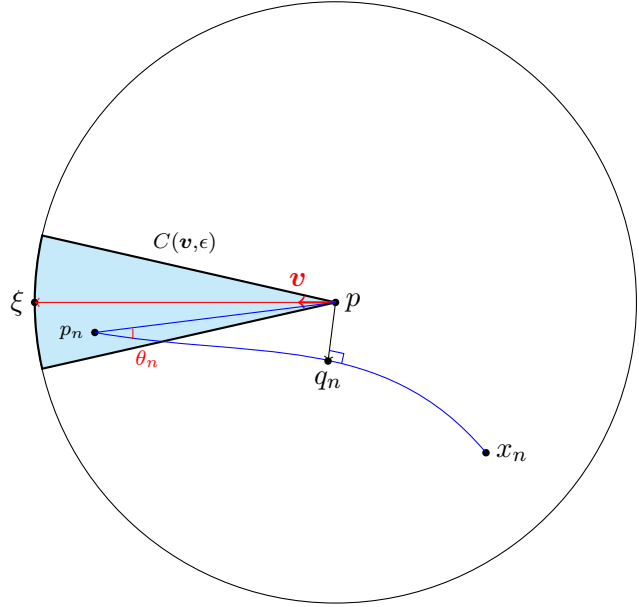

	As $p_n\to\xi$ under the cone topology, we can assume that $\measuredangle_p(p_n,\xi)<\epsilon/2$ for all $n$ because eventually $p_n\in C(\bm{v},\epsilon)$. Then for $x_n\notin C(\bm{v},\epsilon)$, we have
	\begin{displaymath}
		\measuredangle_{p}(p_n,x_n)\geq\frac \epsilon2>0.
	\end{displaymath}

	By the uniform visibility condition, we can find some $R=R(p,\epsilon/2)>0$, such that
	\begin{displaymath}
		d(p,c_{x_n,p_n})=d(p,q_n)\leq R\quad q_n\in c_{x_n,p_n}.
	\end{displaymath}

	Since $d(p,p_n)\to +\infty$, we have
	\begin{displaymath}
		d(p_n,q_n)\geq d(p,p_n)-d(p,q_n)\to\infty,
	\end{displaymath}
	and by the triangle inequality, the distance from $p_n$ to the connecting geodesic segment $c_{p,q_n}$ will goes to infinity. Again, using the uniform visibility condition, we have
	\begin{displaymath}
		\theta_n=\measuredangle_{p_n}(p,q_n)=\measuredangle_{p_n}(p,x_n)\to 0.
	\end{displaymath}
\end{proof}

Based on the Lemma~\ref{lem_7_3}, we are ready to characterize the non-wandering set under our setting.

\begin{theorem}\label{thm_7_4}
	Under the same setting of Lemma~\ref{lem_7_3}, $\bm{v}\in\Omega$ if and only if $c_{\bm{\widetilde{v}}}(\pm\infty)\in L(\Gamma)$. Here $\widetilde{v}$ is any lift of $v$ in $T^1\widetilde{M}$.
\end{theorem}

\begin{proof}
	We start the proof by showing that the $\Gamma$-action on the limit set $L(\Gamma)$ is minimal.

	Fix $\xi\in L(\Gamma)$, and choose any point $\eta\in L(\Gamma)$ with $V$ being an open neighborhood in $\overline{\widetilde{M}}$ of $\eta$. By the definition of limit points we can find a sequence of $\{\alpha_n\}\subset\Gamma$ and $p\in\widetilde{M}$, such that
	\begin{displaymath}
		\alpha_n p\to\eta,\quad n\to\infty.
	\end{displaymath}

	Because $\Gamma$ is discrete, by passing to a subsequence if needed, we can assume $\alpha_n^{-1}p$ is a converging sequence, say it converges to some point $\zeta\in\widetilde{M}(\infty)$. $\Gamma$ is non-elementary, we can find some $\beta\in\Gamma$ such that $\beta(\xi)\neq\zeta$. In particular, we can find an open neighborhood $U\subset\overline{\widetilde{M}}$ of $\zeta$ with $\beta(\xi)\notin U$.

	By Proposition~\ref{prop_2_7}~(3), for any $q\in\widetilde{M}$, we also have
	\begin{displaymath}
    \alpha_n q\to\eta,\quad \alpha_n^{-1} q\to\zeta.
	\end{displaymath}

	Without loss of generality, choose $p\notin U$. Denote $\bm{v}=c'_{p,\eta}(0)$ the unit tangent vector on the connecting geodesic of $p$ to $\eta$. Under the cone topology, we can assume the neighborhood $V$ of $\eta$ is chosen to be a cone with vertex $p$, denoted by $V=C(\bm{v},\epsilon)$ for some $\epsilon>0$. For the isometry $\alpha_n$, we have
	\begin{displaymath}
		\measuredangle_p(\alpha_n(\overline{\widetilde{M}}-U))=\measuredangle_{\alpha_n^{-1}p}(\overline{\widetilde{M}}-U).
	\end{displaymath}

	But $\alpha_n^{-1}p\to\zeta\in U$, this forces $\measuredangle_{\alpha_n^{-1}p}(\overline{\widetilde{M}}-U)\to0$ by Lemma~\ref{lem_7_3}.

	Note that $p\notin U$ suggests that $\alpha_n p\in\alpha_n(\overline{\widetilde{M}}-U)$. Therefore, we have
	\begin{displaymath}
		\max_{y\in\alpha_n(\overline{\widetilde{M}}-U)}\measuredangle_p(\alpha_n p,y)\to0,\quad n\to\infty.
	\end{displaymath}

	But $\alpha_n p\to\eta\in V$, by the cone topology, we can conclude that for sufficiently large $n$,
	\begin{displaymath}
		\alpha_n(\overline{\widetilde{M}}-U)\subset V.
	\end{displaymath}

    Similarly, we can get $\alpha_n^{-1}(\overline{\widetilde{M}}-V)\subset U$.

	So $\alpha_n\beta(\xi)\in V$ as $\beta(\xi)\notin U$. As $\eta$ and $V$ is chosen arbitrarily, the $\Gamma$-orbit of $\beta(\xi)$ is dense in $L(\Gamma)$. The minimality follows the choice of $\xi$ is arbitrary.

	Next, for $\xi\in\widetilde{M}(\infty)$, denote
	\begin{displaymath}
		D(\xi)=\{\eta\in\widetilde{M}(\infty)\mid\exists \alpha_n\in\Gamma,~p\in\widetilde{M},~{s.t.}~\alpha_n^{-1}p\to\xi, \alpha_n p\to\eta\}.
	\end{displaymath}

	By the definition of $D(\xi)$, it is a $\Gamma$-invariant closed set in $L(\Gamma)$.

	\textbf{Claim}: for any limit point $\xi\in L(\Gamma)$, we have $D(\xi)=L(\Gamma)$.

	Given any $\xi\in L(\Gamma)$, choose $\zeta\in D(\xi)\subset L(\Gamma)$, by the minimality of the $\Gamma$-action on the limit set, $\overline{\Gamma(\zeta)}=L(\Gamma)$. We also have that $D(\xi)$ is closed and invariant under $\Gamma$. Therefore,
	\begin{displaymath}
		L(\Gamma)=\overline{\Gamma(\zeta)}\subseteq D(\xi)\subseteq L(\Gamma).
	\end{displaymath}

    Lastly, the Lemma 2.2 in~\cite{LZ} shows that for any pair of vectors $\bm{v}$ and $\bm{w}$ in $T^1 M$, with $\widetilde{\bm{v}}$ and $\widetilde{\bm{w}}$ being any lifts of them in $T^1\widetilde{M}$, $\bm{w}\in P^{+}(\bm{v})$ if and only if $c_{\widetilde{\bm{w}}}(-\infty)\in D(c_{\widetilde{\bm{v}}}(+\infty))$.

	Now we go back to the proof of the statement. Assume that $\bm{v}\in\Omega$, then $\bm{v}\in P^{+}(\bm{v})$, then for any lift $\widetilde{\bm{v}}$, we have
	\begin{displaymath}
		c_{\widetilde{\bm{v}}}(-\infty)\in D(c_{\widetilde{\bm{v}}}(+\infty)).
	\end{displaymath}

	Thus, both $c_{\widetilde{\bm{v}}}(-\infty)$ and $c_{\widetilde{\bm{v}}}(+\infty)$ are in $L(\Gamma)$ as required.

	For the other direction, suppose both $c_{\widetilde{\bm{v}}}(\pm\infty)\in L(\Gamma)$, from the claim above, we have
	\begin{displaymath}
		c_{\widetilde{\bm{v}}}(-\infty)\in D(c_{\widetilde{\bm{v}}}(+\infty)).
	\end{displaymath}

	Therefore, $\bm{v}\in P^{+}(\bm{v})$ is a non-wandering point.
\end{proof}

If we remove the requirement of the uniform visibility, we can get a weaker result might be useful, and we state it here:

\begin{proposition}\label{prop_7_5}
	Let $M$ be a complete, visibility manifold with no focal points, then $\bm{v}\in\Omega$ if and only if both of the end points of the geodesic flow $c_{\bm{\widetilde{v}}}(\pm\infty)\in L(\Gamma)$, and we can find $\{\alpha_n\}\subset\Gamma$ and $p\in\widetilde{M}$, such that $\alpha_n p\to c_{\bm{\widetilde{v}}}(+\infty)$ and $\alpha_n^{-1}p\to c_{\bm{\widetilde{v}}}(-\infty)$.
\end{proposition}

The proof is almost identical to the proof of Theorem~\ref{thm_7_4} but not using the Lemma~\ref{lem_7_3} that requires the uniform visibility. We will omit it here.

In~\cite{Ki}, based on Eberlein's Theorem~\ref{thm_7_2}, Kim proved that on visibility manifolds of non-positive curvature, $\phi_t$ is conservative implies $\phi_t$ is ergodic, with respect to the BMS measure. We are able to show the same result on uniform visibility manifolds with no focal points based on Theorem~\ref{thm_7_4}.

The rest of the proofs in Section~\ref{sec6} are still valid under our new settings. As a summary, we get that:

\begin{theorem}\label{thm_7_6}
    Let $M$ be a complete uniform visibility manifold with no focal points, $\mu={\{\mu_q\}}_{q\in\widetilde{M}}$ is the Patterson-Sullivan measure, $p\in\widetilde{M}$ is arbitrarily chosen and $m$ is the BMS measure. If the width of the flat strips have a positive lower bound, we have the HTS dichotomy as stated in Theorem~\ref{thm_6_8} and Theorem~\ref{thm_6_9} are still valid.
\end{theorem}

\section*{Acknowledgments}

Fei Liu is partially supported by Natural Science Foundation of Shandong Province under Grant No.~ZR2020MA017.

Fang Wang is partially supported by Natural Science Foundation of China (NSFC) under Grant No.~11871045.

\bibliographystyle{amsplain}
\bibliography{manuscript}
\end{document}